\pdfoutput=1
\RequirePackage{ifpdf}
\ifpdf 
\documentclass[pdftex]{sigma}
\else
\documentclass{sigma}
\fi

\def\Z{\mathbb{Z}}
\def\R{\mathbb{R}}
\def\C{\mathbb{C}}
\def\N{\mathbb{N}}
\def\W{\mathbb{W}}
\def\I{\mathbb{I}}
\def\P{\mathbb{P}}
\def\E{\mathbb{E}}
\def\L{\mathbb{L}}
\def\K{\mathbb{K}}
\def\u{{\bf u}}
\def\x{{\bf x}}
\def\y{{\bf y}}
\def\z{{\bf z}}
\def\v{{\bf v}}
\def\B{{\bf B}}
\def\X{{\bf X}}
\def\t{{\bf t}}
\def\mbg{{\bf g}}
\def\bE{{\bf E}}
\def\bK{{\bf K}}
\def\g{{\bf g}}
\def\r{{\bf r}}
\def\cB{{\cal B}}
\def\cF{{\cal F}}
\def\cN{{\cal N}}
\def\cS{{\cal S}}
\def\rP{{\rm P}}
\def\rE{{\rm E}}
\def\rC{{\rm C}}
\def\rc{{\rm c}}
\def\sff{{\sf f}}
\def\RN{{R}}
\def\AN{{A}_{N-1}}
\def\BN{{B}_N}
\def\BNv{{B}^{\vee}_N}
\def\CN{{C}_N}
\def\CNv{{C}^{\vee}_N}
\def\BCN{{BC}_N}
\def\DN{{D}_N}

\numberwithin{equation}{section}

\newtheorem{Theorem}{Theorem}[section]
\newtheorem{Corollary}[Theorem]{Corollary}
\newtheorem{Lemma}[Theorem]{Lemma}
\newtheorem{Proposition}[Theorem]{Proposition}

\begin{document}

\allowdisplaybreaks

\newcommand{\arXivNumber}{1703.03914}

\renewcommand{\thefootnote}{}

\renewcommand{\PaperNumber}{079}

\FirstPageHeading

\ShortArticleName{Elliptic Determinantal Processes and Elliptic Dyson Models}

\ArticleName{Elliptic Determinantal Processes \\ and Elliptic Dyson Models\footnote{This paper is a~contribution to the Special Issue on Elliptic Hypergeometric Functions and Their Applications. The full collection is available at \href{https://www.emis.de/journals/SIGMA/EHF2017.html}{https://www.emis.de/journals/SIGMA/EHF2017.html}}}

\Author{Makoto KATORI}

\AuthorNameForHeading{M.~Katori}

\Address{Department of Physics, Faculty of Science and Engineering, Chuo University,\\
Kasuga, Bunkyo-ku, Tokyo 112-8551, Japan}
\Email{\href{mailto:katori@phys.chuo-u.ac.jp}{katori@phys.chuo-u.ac.jp}}

\ArticleDates{Received April 19, 2017, in f\/inal form September 29, 2017; Published online October 04, 2017}

\Abstract{We introduce seven families of stochastic systems of interacting particles in one-dimension corresponding to the seven families of irreducible reduced af\/f\/ine root systems. We prove that they are determinantal in the sense that all spatio-temporal correlation functions are given by determinants controlled by a single function called the spatio-temporal correlation kernel. For the four families ${A}_{N-1}$, ${B}_N$, ${C}_N$ and ${D}_N$, we identify the systems of stochastic dif\/ferential equations solved by these determinantal processes, which will be regarded as the elliptic extensions of the Dyson model. Here we use the notion of martingales in probability theory and the elliptic determinant evaluations of the Macdonald denominators of irreducible reduced af\/f\/ine root systems given by Rosengren and Schlosser.}


\Classification{60J65; 60G44; 82C22; 60B20; 33E05; 17B22}

\renewcommand{\thefootnote}{\arabic{footnote}}
\setcounter{footnote}{0}

\section{Introduction and main results} \label{sec:Introduction}

Stochastic analysis on interacting particle systems is important to provide useful models describing equilibrium and non-equilibrium phenomena studied in statistical physics \cite{Kat16_Springer}. {\it Determinantal process} is a stochastic system of interacting particles which is integrable in the sense that all spatio-temporal correlation functions are given by determinants controlled by a single function called the {\it spatio-temporal correlation kernel} \cite{BR05,KT07b}. Since the generating functions of correlation functions are generally given by the Laplace transforms of probability densities, the stochastic integrability of determinantal processes is proved by showing that the Laplace transform of any multi-time joint probability density is expressed by the {\it spatio-temporal Fredholm determinant} associated with the correlation kernel. The purpose of this paper is to present new kinds of determinantal processes in which the interactions between particles are described by the logarithmic derivatives of Jacobi's theta functions. A classical example of determinantal processes is Dyson's Brownian motion model with parameter $\beta=2$, which is a dynamical version of the eigenvalue statistics of random matrices in the Gaussian unitary ensemble (GUE), and we call it simply the {\it Dyson model} \cite{Dys62,Kat16_Springer,Spo87}. We will extend the Dyson model to the elliptic-function-level in this paper. We use the notion of {\it martingales} in probability theory \cite{Kat14,Kat16_Springer} and the {\it elliptic determinantal evaluations of the Macdonald denominators} of seven families of irreducible reduced af\/f\/ine root systems given by Rosengren and Schlosser \cite{RS06} (see also \cite{Kra05,War02}).

Among the seven families of irreducible reduced af\/f\/ine root systems, $\RN=\AN, \BN, \BNv, \CN$, $\CNv, \BCN$, and $\DN$, we reported the results only for the system $\RN=\AN$ in the previous papers \cite{Kat15,Kat16} as follows. Assume $0< t_{\ast} < \infty$, $0< r < \infty$, $0< \cN < \infty$, and let
\begin{gather}
A_{\cN}^{2 \pi r}(t_{\ast}-t,x) = \left[ \frac{1}{2 \pi r} \frac{\partial}{\partial v}
\log \vartheta_1(v; \tau) \right]_{v=x/2 \pi r,\, \tau=i \cN (t_{\ast}-t)/2 \pi r^2}, \qquad t \in [0, t_{\ast}), \label{eqn:A1}
\end{gather}
where $\theta_1(v; \tau)$ denotes one of the Jacobi theta functions. See Appendix~\ref{sec:appendixA} for the Jacobi theta functions and related functions. For $N \in \{2,3, \dots\}$, we def\/ine the Weyl chamber
\begin{gather*}
\W_N = \big\{ \x=(x_1, x_2, \dots, x_N) \in \R^N\colon x_1 < x_2 < \cdots < x_N \big\}.
\end{gather*}
We introduced a one-parameter ($\beta > 0$) family of systems of stochastic dif\/ferential equations (SDEs) for $t \in [0, t_{\ast})$~\cite{Kat16}
\begin{alignat}{3}
& (\AN) \quad && X^{\AN}_j(t) = u_j+W_j(t)
+ \frac{\beta}{2} \int_0^t A^{2 \pi r}_{N}
\left(t_{\ast}-s, \sum_{\ell=1}^N X^{\AN}_{\ell}(s)-\kappa_N \right) ds& \nonumber\\
& && \hphantom{X^{\AN}_j(t) =}{} + \frac{\beta}{2} \sum_{\substack{ 1 \leq k \leq N, \\ k \not=j}}
\int_0^t A^{2 \pi r}_{N}\big(t_{\ast}-s, X^{\AN}_j(s)-X^{\AN}_k(s)\big) ds,& \nonumber\\
&&& j=1, \dots, N, \ \ \text{in $\R$},\label{eqn:SDE_A1}
\end{alignat}
for $\u=(u_1, \dots, u_N) \in \W_N$, where $W_j(t)$, $t \geq 0$, $j =1, \dots, N$ are independent one-dimensional standard Brownian motions, and
\begin{gather*}
\kappa_N= \begin{cases}
\pi r (N-1), & \text{if $N$ is even}, \\
\pi r (N-2), & \text{if $N$ is odd}.
\end{cases}
\end{gather*}
We called this family of $N$-particle systems on $\R$, $\X^{\AN}(t)=\big(X^{\AN}_1(t), \dots, X^{\AN}_N(t)\big)$, $t \in [0, t_{\ast})$, the {\it elliptic Dyson model of type~A} with parameter~$\beta$. By~(\ref{eqn:A_t_ast}) in Appendix~\ref{sec:appendixA_2}, (\ref{eqn:SDE_A1})~gives
\begin{gather*}
d X^{\AN}_j(t) \sim d W_j(t) + \frac{\beta}{2} \frac{v^{\AN}_j - X^{\AN}_j(t)}{t_{\ast}-t} dt, \qquad j=1,2, \dots, N,
\quad \text{in $t \uparrow t_{\ast}$},
\end{gather*}
with
\begin{gather*}
v^{\AN}_j = \begin{cases}
\displaystyle \frac{\pi r}{N}(2j-1), & \text{if $N$ is even}, \vspace{1mm}\\
\displaystyle \frac{2 \pi r}{N}(j-1), & \text{if $N$ is odd},
\end{cases} \qquad j=1,2, \dots, N,
\end{gather*}
which give equidistant-spacing conf\/igurations $\v^{\AN}=\big(v^{\AN}_1, \dots, v^{\AN}_N\big)$. This implies that the elliptic Dyson model of type~A is realized as a system of interacting {\it Brownian bridges} (see, for instance, \cite[Part~I, Section~IV.4.22]{BS02}) pinned at the conf\/iguration $\v^{\AN}$ at time~$t_{\ast}$~\cite{Kat16}. When the system of SDEs is temporally homogeneous, it is well known that the corresponding Kolmogorov (Fokker--Planck) equation for transition probability density can be transformed into a {\it Calogero--Moser--Sutherland quantum system} (see, for instance, \cite[Chapter~11]{For10}). The present system of interacting Brownian bridges is temporally inhomogeneous, however, and the Kolmogorov equation is mapped into a Schr\"odinger-type equation with {\it time-dependent} Hamiltonian and the {\it time-dependent} ground energy~\cite{Kat16}. The obtained quantum system is elliptic, but dif\/ferent from the {\it elliptic Calogero--Moser--Sutherland model} extensively studied as a quantum integrable system \cite{FV97,KW00,OP83,Sut75,Tak00}. We found at the same time that the interaction among particles vanishes when the parameter is chosen to be a special value ($\beta=2$ in our case) as found in the usual Calogero--Moser--Sutherland models. We applied the determinantal-martingale-me\-thod~\cite{Kat14} and proved that the elliptic Dyson model of type~A with $\beta=2$ is an {\it integrable stochastic process} in a sense that it is determinantal for a set of observables~\cite{Kat15,Kat16}.

In the present paper, we report the results for other six systems, $\RN=\BN, \BNv, \CN, \CNv$, $\BCN, \DN$. Here we f\/irst construct the six families of determinantal processes (Theorem~\ref{thm:main1}). Then for the three families $\BN$, $\CN$ and $\DN$, we clarify the systems of SDEs (with parameter $\beta=2$) which are solved by our new determinantal processes (Theorem~\ref{thm:main2}).

For $N \in \{2,3, \dots\}$, $0 < r < \infty$, def\/ine
\begin{gather*}
\W_N^{(0, \pi r)} =\big\{\x =(x_1, x_2, \dots, x_N) \in \R^N\colon 0 < x_1 < x_2 < \cdots < x_N < \pi r \big\}.
\end{gather*}
Let $\u=(u_1, u_2, \dots, u_N) \in \W_N^{(0, \pi r)}$, and
\begin{gather}
\tau^{\RN}(t)=\frac{i \cN^{\RN} (t_{\ast}-t)}{2 \pi r^2},\label{eqn:tau_R}
\end{gather}
where the numbers $\cN^{\RN}$ are given by
\begin{gather}
\cN^{\RN} = \begin{cases}
N, & \RN = \AN, \\
2N-1, & \RN = \BN, \\
2N,& \RN = \BNv, \CNv, \\
2(N+1),& \RN =\CN, \\
2N+1,& \RN = \BCN, \\
2(N-1),& \RN = \DN.
\end{cases}\label{eqn:cN_R}
\end{gather}
For $t \in [0, t_{\ast})$ put
\begin{gather}
D^{\RN}_{\u}(t, \x)= \frac{c^{\RN}_0\big(\tau^{\RN}(t)\big)}{c^{\RN}_0\big(\tau^{\RN}(0)\big)}\prod_{\ell=1}^N
\frac{\vartheta_1\big(c^{\RN}_1 x_{\ell}/ 2\pi r; c^{\RN}_2 \tau^{\RN}(t)\big)}{\vartheta_1\big(c^{\RN}_1 u_{\ell}/ 2\pi r; c^{\RN}_2 \tau^{\RN}(0)\big)}
\nonumber\\
\hphantom{D^{\RN}_{\u}(t, \x)=}{} \times
\prod_{1 \leq j < k \leq N} \frac{\vartheta_1\big((x_k-x_j)/2 \pi r; \tau^{\RN}(t)\big)}
{\vartheta_1\big((u_k-u_j)/2 \pi r; \tau^{\RN}(0)\big)}\frac{\vartheta_1\big((x_k+x_j)/2 \pi r; \tau^{\RN}(t)\big)}{\vartheta_1\big((u_k+u_j)/2 \pi r; \tau^{\RN}(0)\big)}
\nonumber\\
\hphantom{D^{\RN}_{\u}(t, \x)=}{}\mbox{for $\RN=\BN, \BNv, \CN, \CNv$},\label{eqn:D_Ra}
\\
D^{\BCN}_{\u}(t, \x)= \frac{c^{\BCN}_0\big(\tau^{\BCN}(t)\big)}{c^{\BCN}_0\big(\tau^{\BCN}(0)\big)}\prod_{\ell=1}^N
\frac{\vartheta_1\big(x_{\ell}/ 2\pi r; \tau^{\BCN}(t)\big)}{\vartheta_1\big(u_{\ell}/ 2\pi r; \tau^{\BCN}(0)\big)}
\frac{\vartheta_0\big(x_{\ell}/ \pi r; 2 \tau^{\BCN}(t)\big)}{\vartheta_0\big(u_{\ell}/ \pi r; 2 \tau^{\BCN}(0)\big)}\nonumber\\
\hphantom{D^{\BCN}_{\u}(t, \x)=}{} \times
\prod_{1 \leq j < k \leq N}\frac{\vartheta_1\big((x_k-x_j)/2 \pi r; \tau^{\BCN}(t)\big)}{\vartheta_1\big((u_k-u_j)/2 \pi r; \tau^{\BCN}(0)\big)}
\frac{\vartheta_1\big((x_k+x_j)/2 \pi r; \tau^{\BCN}(t)\big)}{\vartheta_1\big((u_k+u_j)/2 \pi r; \tau^{\BCN}(0)\big)},\label{eqn:D_Rb}
\\
D^{\DN}_{\u}(t, \x)= \frac{c^{\DN}_0\big(\tau^{\DN}(t)\big)}{c^{\DN}_0\big(\tau^{\DN}(0)\big)}\nonumber\\
\hphantom{D^{\DN}_{\u}(t, \x)=}{} \times
\prod_{1 \leq j < k \leq N} \frac{\vartheta_1\big((x_k-x_j)/2 \pi r; \tau^{\DN}(t)\big)}{\vartheta_1\big((u_k-u_j)/2 \pi r; \tau^{\DN}(0)\big)}
\frac{\vartheta_1\big((x_k+x_j)/2 \pi r; \tau^{\DN}(t)\big)}{\vartheta_1\big((u_k+u_j)/2 \pi r; \tau^{\DN}(0)\big)},\label{eqn:D_Rc}
\end{gather}
where
\begin{gather}
c^{\RN}_0(\tau)= \eta(\tau)^{-N(N-1)} \qquad \text{for $\RN=\BN, \CN$},\nonumber\\
c^{\BNv}_0(\tau)= \eta(\tau)^{-(N-1)^2} \eta(2 \tau)^{-(N-1)},\nonumber\\
c^{\CNv}_0(\tau)= \eta(\tau)^{-(N-1)^2} \eta(\tau/2)^{-(N-1)},\nonumber\\
c^{\BCN}_0(\tau)= \eta(\tau)^{-N(N-1)} \eta(2 \tau)^{-N},\nonumber\\
c^{\DN}_0(\tau)= \eta(\tau)^{-N(N-2)},\label{eqn:c_0}
\end{gather}
and
\begin{gather}
c^{\BN}_1 =c^{\CNv}_1=1, \qquad c^{\BNv}_1 =c^{\CN}_1=2,\nonumber\\
c^{\BN}_2 =c^{\CN}_2=1, \qquad c^{\BNv}_2 =2, \qquad c^{\CNv}_2 =1/2.\label{eqn:c_1_2_R}
\end{gather}
Here $\eta(\tau)$ denotes the Dedekind modular function (see, for instance,~\cite[Section~23.15]{NIST10})
\begin{gather}
\eta(\tau)=e^{\tau \pi i/12} \prod_{n=1}^{\infty}\big(1-e^{2 n \tau \pi i}\big), \qquad \Im \tau > 0.\label{eqn:Dedekind1}
\end{gather}
We call (\ref{eqn:D_Ra})--(\ref{eqn:D_Rc}) the {\it determinantal martingale-functions}~\cite{Kat14} (see Sections~\ref{sec:basic} and~\ref{sec:martingale}).

In the interval $[0, \pi r]$, we consider the $N$-particle system of one-dimensional standard Brownian motions $\B(t)=(B_1(t), \dots, B_N(t))$, $t \geq 0$ started at $\u=(u_1, \dots, u_N) \in \W_N^{(0, \pi r)}$ with either an absorbing or ref\/lecting boundary condition at the endpoints of the interval, $0$ and $\pi r$. The {\it transition probability density} of each particle is generally denoted by $p^{[0, \pi r]}$. The boundary conditions at 0 and $\pi r$ are indicated by~$b$ and~$b'$, respectively, and the transition probability density with the specif\/ied boundary conditions $b$, $b'$ is written as $p^{[0, \pi]}_{b,b'}$. In this paper the absorbing (resp.\ ref\/lecting) boundary condition is abbreviated as~`a' (resp.~`r'). By the ref\/lection principle of Brownian motion (see, for instance, \cite[Appendixes~1.5 and~1.6]{BS02}), if both boundaries are absorbing, the transition probability density is given by
\begin{gather}
p^{[0, \pi r]}_{\rm aa}(t, y|x) = \sum_{k=-\infty}^{\infty} \big\{ p_{\rm BM}(t, y+2 \pi r k| x) - p_{\rm BM}(t, y+2 \pi r k | -x) \big\},\label{eqn:p_abs}
\end{gather}
and if both are ref\/lecting, it is given by
\begin{gather}
p^{[0, \pi r]}_{\rm rr}(t, y|x) = \sum_{k=-\infty}^{\infty} \big\{ p_{\rm BM}(t, y+2 \pi r k| x) + p_{\rm BM}(t, y+2 \pi r k | -x) \big\},\label{eqn:p_ref}
\end{gather}
for $x, y \in [0, \pi r]$, $t \geq 0$, where $p_{\rm BM}(t, y|x)$ denotes the transition probability density of the one-dimensional standard Brownian motion
\begin{gather}
p_{\rm BM}(t, y|x) = \begin{cases}
\displaystyle \frac{e^{-(x-y)^2/2t}}{\sqrt{2 \pi t}}, & t > 0, \\
\delta(x-y), & t =0.
\end{cases} \label{eqn:p_BM}
\end{gather}
We write the probability law of such a system of boundary-conditioned Brownian motions in $[0, \pi r]$ as $\rP^{[0, \pi r]}_{\u}$. In $\rP^{[0, \pi r]}_{\u}$, put
\begin{gather}
T_{\rm collision}=\inf \big\{ t > 0 \colon B_j(t)=B_k(t) \ \text{for any $j \not= k$} \big\},\label{eqn:T1}
\end{gather}
i.e., the f\/irst collision-time of the $N$-particle system of Brownian motions in the interval $[0, \pi r]$. Let ${\bf 1}(\omega)$ be the indicator function of a condition $\omega$; ${\bf 1}(\omega)=1$ if $\omega$ is satisf\/ied, and ${\bf 1}(\omega)=0$ otherwise. Then we def\/ine
\begin{gather}
\P_{\u}^{\RN} \big|_{\cF_t} = {\bf 1}(T_{\rm collision} > t) D^{\RN}_{\u}(t, \B(t)) \rP_{\u}^{[0, \pi r]} \big|_{\cF_t},
\qquad t \in [0, t_{\ast}), \label{eqn:Plaw_R1}
\end{gather}
where $D^{\RN}_{\u}$ are given by (\ref{eqn:D_Ra})--(\ref{eqn:D_Rc}) and $\cF_t$ denotes the f\/iltration associated with the Brownian motion (see Section~\ref{sec:notion_martingale}). That is, the Radon--Nikodym derivative of $\P_{\u}^{\RN}$ with respect to the Wiener measure $\rP_{\u}^{[0, \pi r]}$ of Brownian motion is given by ${\bf 1}(T_{\rm collision} > t) D^{\RN}_{\u}(t, \B(t))$ at each time $t \in [0, t_{\ast})$. Therefore, a~stochastic process with~$N$ particles governed by $\P_{\u}^{\RN}$ is well-def\/ined as a~realization of non-intersecting paths which are absolutely continuous to the $N$-particle paths of independent Brownian motions in $[0, \pi r]$.

For $y \in \R$, $\delta_y(\cdot)$ denotes the delta measure such that $\delta_y(\{x\}) = 1$ if $x=y$ and $\delta_y(\{x\})=0$ otherwise. The f\/irst theorem of this paper is the following.

\begin{Theorem}\label{thm:main1} Assume that $0< t_{\ast} < \infty$, $0 < r < \infty$. For each $N \in \{2, 3, \dots\}$, $\RN=\BN$, $\BNv, \CN, \CNv, \BCN,\DN$, $\u=(u_1, \dots, u_N) \in \W_N^{(0, \pi r)}$, $\P_{\u}^{\RN}$ defined by~\eqref{eqn:Plaw_R1} gives a probability measure and defines a measure-valued stochastic process
\begin{gather}
\Xi^{\RN}(t, \cdot)=\sum_{j=1}^N \delta_{X^{\RN}_j(t)}(\cdot), \qquad t \in [0, t_{\ast}).\label{eqn:Xi_R1}
\end{gather}
The process $\big(\big(\Xi^{\RN}(t)\big)_{t \in [0, t_{\ast})}, \P_{\u}^{\RN}\big)$ is determinantal with the spatio-temporal correlation kernel
\begin{gather}
\K^{\RN}_{\u}(s, x; t, y)= \sum_{j=1}^N p^{[0, \pi r]} (s, x|u_j) M^{\RN}_{\u, u_j}(t, y) -{\bf 1}(s>t) p^{[0, \pi r]}(s-t, x|y),\label{eqn:K1}
\end{gather}
$(s, x), (t, y) \in [0, t_{\ast}) \times [0, \pi r]$, where
\begin{gather}
M^{\RN}_{\u, u_j}(t, x)= \int_{-\infty}^{\infty} \Phi^{\RN}_{\u, u_j}(x+i \widetilde{x})p_{\rm BM}(t, \widetilde{x}|0) d \widetilde{x}\label{eqn:MB1}
\end{gather}
with the sets of entire functions $($the elliptic Lagrange interpolation functions$)$
\begin{gather}
\Phi^{\RN}_{\u, u_j}(z)= \frac{\vartheta_1\big(c^{\RN}_1 z/2 \pi r; c^{\RN}_2 \tau^{\RN}(0)\big)}
{\vartheta_1\big(c^{\RN}_1 u_j/2 \pi r; c^{\RN}_2 \tau^{\RN}(0)\big)}\nonumber\\
\hphantom{\Phi^{\RN}_{\u, u_j}(z)=}{} \times
\prod_{\substack{1 \leq \ell \leq N, \cr \ell \not=j}}
\frac{\vartheta_1\big((z-u_{\ell})/2 \pi r; \tau^{\RN}(0)\big)}
{\vartheta_1\big((u_j-u_{\ell})/2 \pi r; \tau^{\RN}(0)\big)}
\frac{\vartheta_1\big((z+u_{\ell})/2 \pi r; \tau^{\RN}(0)\big)}
{\vartheta_1\big((u_j+u_{\ell})/2 \pi r; \tau^{\RN}(0)\big)}
\nonumber\\
\hphantom{\Phi^{\RN}_{\u, u_j}(z)=}{} \text{for $\RN=\BN, \BNv, \CN, \CNv$},\label{eqn:Phi_Ra}\\
\Phi^{\BCN}_{\u, u_j}(z) =\frac{\vartheta_1\big(z/2 \pi r; \tau^{\BCN}(0)\big)}{\vartheta_1\big(u_j/2 \pi r; \tau^{\BCN}(0)\big)}
\frac{\vartheta_0\big(z/\pi r; 2 \tau^{\BCN}(0)\big)}{\vartheta_0\big(u_j/\pi r; 2 \tau^{\BCN}(0)\big)}\nonumber\\
\hphantom{\Phi^{\BCN}_{\u, u_j}(z) =}{} \times
\prod_{\substack{1 \leq \ell \leq N, \cr \ell \not=j}}\frac{\vartheta_1\big((z-u_{\ell})/2 \pi r; \tau^{\BCN}(0)\big)}
{\vartheta_1\big((u_j-u_{\ell})/2 \pi r; \tau^{\BCN}(0)\big)}\frac{\vartheta_1\big((z+u_{\ell})/2 \pi r; \tau^{\BCN}(0)\big)}{\vartheta_1\big((u_j+u_{\ell})/2 \pi r; \tau^{\BCN}(0)\big)},\label{eqn:Phi_Rb}
\\
\Phi^{\DN}_{\u, u_j}(z) =\prod_{\substack{1 \leq \ell \leq N, \cr \ell \not=j}}\frac{\vartheta_1\big((z-u_{\ell})/2 \pi r; \tau^{\DN}(0)\big)}
{\vartheta_1\big((u_j-u_{\ell})/2 \pi r; \tau^{\DN}(0)\big)}\frac{\vartheta_1\big((z+u_{\ell})/2 \pi r; \tau^{\DN}(0)\big)}
{\vartheta_1\big((u_j+u_{\ell})/2 \pi r; \tau^{\DN}(0)\big)},\label{eqn:Phi_Rc}
\end{gather}
$j=1,2, \dots, N$. Here $\tau^{\RN}(t)$ are given by \eqref{eqn:tau_R} with \eqref{eqn:cN_R}, and $c^{\RN}_1$, $c^{\RN}_2$ for $\RN=\BN, \BNv, \CN, \CNv$ are given by~\eqref{eqn:c_1_2_R}.
\end{Theorem}

The second theorem of this paper is the following.
\begin{Theorem}\label{thm:main2}
Let $N \in \{2, 3, \dots\}$, $0< t_{\ast} < \infty$, $0 < r < \infty$. Assume $\u=(u_1, \dots, u_N) \in \W_N^{(0, \pi r)}$. In the interval $[0, \pi r]$, for $\big(\big(\Xi^{\BN}(t)\big)_{t \in [0, t_{\ast})}, \P^{\BN}_{\u}\big)$ we put an absorbing boundary condition at~$0$ and a~reflecting boundary condition at $\pi r$, for $\big(\big(\Xi^{\CN}(t)\big)_{t \in [0, t_{\ast})}, \P^{\CN}_{\u}\big)$ we put an absorbing boundary condition both at~$0$ and~$\pi r$, and for $\big(\big(\Xi^{\DN}(t)\big)_{t \in [0, t_{\ast})}, \P^{\DN}_{\u}\big)$ we put a reflecting boundary condition both at~$0$ and~$\pi r$, respectively. If we set $\Xi^{\RN}(t)= \sum\limits_{j=1}^N \delta_{X^{\RN}_j(t)}$, then $\X^{\RN}(t)=\big(X^{\RN}_1(t), \dots, X^{\RN}_N(t)\big)$, $\RN=\BN, \CN, \DN$, solve the following systems of SDEs
\begin{gather}
 (\BN) \quad X^{\BN}_j(t) = u_j+W_j(t) +\int_0^t A^{2 \pi r}_{2N-1}\big(t_{\ast}-s, X^{\BN}_j(s)\big) ds\nonumber\\
\hphantom{(\BN) \quad X^{\BN}_j(t) =}{} + \sum_{\substack{ 1 \leq k \leq N, \cr k \not=j}} \int_0^t \big\{ A^{2 \pi r}_{2N-1}\big(t_{\ast}-s, X^{\BN}_j(s)-X^{\BN}_k(s)\big)\nonumber\\
\hphantom{(\BN) \quad X^{\BN}_j(t) =}{} +A^{2 \pi r}_{2N-1}\big(t_{\ast}-s, X^{\BN}_j(s)+X^{\BN}_k(s)\big) \big\} ds,\qquad j=1,2, \dots, N, \nonumber\\
\hphantom{(\BN) \quad X^{\BN}_j(t) =}{}
\text{in $[0, \pi r]$ with a reflecting boundary condition at $\pi r$},\label{eqn:SDE_B1}
\\
 (\CN) \quad X^{\CN}_j(t) = u_j+W_j(t)+2 \int_0^t A^{2 \pi r}_{2(N+1)}\big(t_{\ast}-s, 2 X^{\CN}_j(s)\big) ds\nonumber\\
\hphantom{(\CN) \quad X^{\CN}_j(t) =}{} + \sum_{\substack{ 1 \leq k \leq N, \cr k \not=j}}
\int_0^t \big\{ A^{2 \pi r}_{2(N+1)}\big(t_{\ast}-s, X^{\CN}_j(s)-X^{\CN}_k(s)\big) \nonumber\\
\hphantom{(\CN) \quad X^{\CN}_j(t) =}{} +A^{2 \pi r}_{2(N+1)}\big(t_{\ast}-s, X^{\CN}_j(s)+X^{\CN}_k(s)\big) \big\} ds, \qquad j=1,2, \dots, N, \nonumber\\
\hphantom{(\CN) \quad X^{\CN}_j(t) =}{} \text{in $[0, \pi r]$},\label{eqn:SDE_C1}
\\
 (\DN) \quad X^{\DN}_j(t) = u_j+W_j(t) + \sum_{\substack{ 1 \leq k \leq N, \cr k \not=j}}
\int_0^t \big\{ A^{2 \pi r}_{2(N-1)}\big(t_{\ast}-s, X^{\DN}_j(s)-X^{\DN}_k(s)\big)\nonumber\\
\hphantom{(\DN) \quad X^{\DN}_j(t) =}{}+A^{2 \pi r}_{2(N-1)}\big(t_{\ast}-s, X^{\DN}_j(s)+X^{\DN}_k(s)\big) \big\} ds,\qquad j=1,2, \dots, N, \nonumber\\
\hphantom{(\DN) \quad X^{\DN}_j(t) =}{}
\text{in $[0, \pi r]$ with a reflecting boundary condition both at $0$ and $\pi r$},\label{eqn:SDE_D1}
\end{gather}
where $W_j(t)$, $t \geq 0$, $j=1,2, \dots, N$ are independent one-dimensional standard Brownian motions.
\end{Theorem}

We call the systems (\ref{eqn:SDE_B1})--(\ref{eqn:SDE_D1}) the {\it elliptic Dyson models of types B, C, D}, respectively. By~(\ref{eqn:A_t_ast}) in Appendix~\ref{sec:appendixA_2}, we see that in $t \uparrow t_{\ast}$
\begin{gather*}
d X^{\RN}_j(t) \sim d W_j(t) + \frac{v^{\RN}_j - X^{\RN}_j(t)}{t_{\ast}-t} dt,\qquad j=1,2, \dots, N, \qquad \RN=\BN, \CN, \DN
\end{gather*}
with
\begin{gather*}
v^{\BN}_j = \frac{2j-1}{2N-1} \pi r, \qquad v^{\CN}_j = \frac{j}{N+1} \pi r, \qquad v^{\DN}_j = \frac{j-1}{N-1} \pi r, \qquad j=1,2, \dots, N.
\end{gather*}
The above implies that the elliptic Dyson models of types B, C, and D are realized as the systems of interacting Brownian bridges pinned at the conf\/igurations $\v^{\RN}=\big(v^{\RN}_1, \dots, v^{\RN}_N\big)$, $\RN=\BN, \CN, \DN$, at time $t_{\ast}$. Note that $v^{\BN}_1 >0$, $v^{\BN}_N=\pi r$;
$v^{\CN}_1 >0$, $v^{\CN}_N < \pi r$; $v^{\DN}_1=0$, $v^{\DN}_N=\pi r$, corresponding to the situation such that the absorbing boundary condition is imposed at~0 for~$\BN$, and at~0 and $\pi r$ for $\CN$, while the ref\/lecting boundary condition is imposed at other endpoints of the interval $[0, \pi r]$.

We see
\begin{gather*}
\lim_{t_{\ast} \to \infty} A_{\cN}^{2 \pi r}(t_{\ast}-t,x) = \frac{1}{2 r} \cot \left( \frac{x}{2r} \right),
\end{gather*}
by (\ref{eqn:A1}) and (\ref{eqn:theta_asym}). Hence in the limit $t_{\ast} \to \infty$, the systems of SDEs (\ref{eqn:SDE_A1}) with $\beta=2$, and of (\ref{eqn:SDE_B1})--(\ref{eqn:SDE_D1}) become the following temporally homogeneous systems of SDEs for $t \in [0, \infty)$:
\begin{gather}
(\AN) \quad X^{\AN}_j(t) = u_j+W_j(t)-\frac{1}{2r} \int_0^t \tan \left( \frac{1}{2r} \sum_{\ell=1}^N X^{\AN}_{\ell}(s) \right) ds\nonumber\\
\hphantom{(\AN) \quad X^{\AN}_j(t) =}{} + \frac{1}{2r} \sum_{\substack{ 1 \leq k \leq N, \cr k \not=j}}
\int_0^t \cot \left( \frac{X^{\AN}_j(s)-X^{\AN}_k(s)}{2 r} \right) ds,\nonumber\\
\hphantom{(\AN) \quad X^{\AN}_j(t) =}{} j=1,2, \dots, N, \quad \text{in $\R$},\label{eqn:SDE_A2}\\
(\BN) \quad X^{\BN}_j(t) = u_j+W_j(t)+ \frac{1}{2r} \int_0^t \cot \left( \frac{X^{\BN}_j(s)}{2r} \right) ds\nonumber\\
\hphantom{(\BN) \quad X^{\BN}_j(t) =}{} + \frac{1}{2r} \sum_{\substack{ 1 \leq k \leq N, \cr k \not=j}}\int_0^t \left\{
\cot \left( \frac{X^{\BN}_j(s)-X^{\BN}_k(s)}{2 r} \right)\right.\nonumber\\
\left. \hphantom{(\BN) \quad X^{\BN}_j(t) =}{}+ \cot \left( \frac{X^{\BN}_j(s)+X^{\BN}_k(s)}{2 r} \right)\right\} ds,\quad j=1,2, \dots, N, \nonumber\\
\hphantom{(\BN) \quad X^{\BN}_j(t) =}{} \text{in $[0, \pi r]$ with a ref\/lecting boundary condition at $\pi r$},
\label{eqn:SDE_B2}\\
 (\CN) \quad X^{\CN}_j(t) = u_j+W_j(t)+ \frac{1}{r} \int_0^t \cot \left( \frac{X^{\CN}_j(s)}{r} \right) ds\nonumber\\
\hphantom{(\CN) \quad X^{\CN}_j(t) =}{} + \frac{1}{2r} \sum_{\substack{ 1 \leq k \leq N, \cr k \not=j}}\int_0^t \left\{
\cot \left( \frac{X^{\CN}_j(s)-X^{\CN}_k(s)}{2 r} \right)\right.\nonumber\\
\left. \hphantom{(\CN) \quad X^{\CN}_j(t) =}{} + \cot \left( \frac{X^{\CN}_j(s)+X^{\CN}_k(s)}{2 r} \right)\right\} ds,\quad
j=1,2, \dots, N, \quad \text{in $[0, \pi r]$},\label{eqn:SDE_C2}\\
 (\DN) \quad X^{\DN}_j(t) = u_j+W_j(t)+ \frac{1}{2r} \sum_{\substack{ 1 \leq k \leq N, \cr k \not=j}}
\int_0^t \left\{\cot \left( \frac{X^{\DN}_j(s)-X^{\DN}_k(s)}{2 r} \right)\right.\nonumber\\
\left.\hphantom{(\DN) \quad X^{\DN}_j(t) =}{} + \cot \left( \frac{X^{\DN}_j(s)+X^{\DN}_k(s)}{2 r} \right)\right\} ds,\quad j=1,2, \dots, N, \nonumber\\
\hphantom{(\DN) \quad X^{\DN}_j(t) =}{} \text{in $[0, \pi r]$ with a ref\/lecting boundary condition both at $0$ and $\pi r$},\label{eqn:SDE_D2}
\end{gather}
respectively. We will call the above systems (\ref{eqn:SDE_A2})--(\ref{eqn:SDE_D2}) the {\it trigonometric Dyson models of types A, B, C and D}, respectively.

Moreover, for
\begin{gather*}
\lim_{r \to \infty} \frac{1}{2r} \cot \left( \frac{x}{2r} \right) = \frac{1}{x},\qquad \lim_{r \to \infty} \frac{1}{2r} \tan \left( \frac{x}{2r} \right)=0,
\end{gather*}
the $r \to \infty$ limit of the systems (\ref{eqn:SDE_A2})--(\ref{eqn:SDE_D2}) are given as follows, for $t \in [0, \infty)$
\begin{gather}
 (\AN) \quad X^{\AN}_j(t) = u_j+W_j(t)+ \sum_{\substack{ 1 \leq k \leq N, \cr k \not=j}}\int_0^t\frac{1}{X^{\AN}_j(s)-X^{\AN}_k(s)} ds,\nonumber\\
\hphantom{(\AN) \quad X^{\AN}_j(t) =}{} j=1,2, \dots, N, \quad \text{in $\R$},\label{eqn:SDE_A3}\\
 (\CN) \quad X^{\CN}_j(t) = u_j+W_j(t)+ \int_0^t \frac{1}{X^{\CN}_j(s)} ds\nonumber\\
\hphantom{(\CN) \quad X^{\CN}_j(t) =}{} +\sum_{\substack{ 1 \leq k \leq N, \cr k \not=j}}\int_0^t \left\{
\frac{1}{X^{\CN}_j(s)-X^{\CN}_k(s)}+ \frac{1}{X^{\CN}_j(s)+X^{\CN}_k(s)}\right\} ds,\nonumber\\
\hphantom{(\CN) \quad X^{\CN}_j(t) =}{}
j=1,2, \dots, N, \quad \text{in $[0, \infty)$},\label{eqn:SDE_C3}\\
 (\DN) \quad X^{\DN}_j(t) = u_j+W_j(t)\nonumber\\
\hphantom{(\DN) \quad X^{\DN}_j(t) =}{} +\sum_{\substack{ 1 \leq k \leq N, \cr k \not=j}}
\int_0^t \left\{\frac{1}{X^{\DN}_j(s)-X^{\DN}_k(s)}+ \frac{1}{X^{\DN}_j(s)+X^{\DN}_k(s)}\right\} ds,\label{eqn:SDE_D3}\\
\hphantom{(\DN) \quad X^{\DN}_j(t) =}{} j=1,2, \dots, N, \quad \text{in $[0, \infty)$ with a ref\/lecting boundary condition at 0}.\nonumber
\end{gather}
That is, the system (\ref{eqn:SDE_A2}) is reduced to (\ref{eqn:SDE_A3}), (\ref{eqn:SDE_B2}) and (\ref{eqn:SDE_C2}) are degenerated to (\ref{eqn:SDE_C3}), and (\ref{eqn:SDE_D2}) to (\ref{eqn:SDE_D3}). The system (\ref{eqn:SDE_A3}) is the original Dyson model (Dyson's Brownian motion model with parameter $\beta=2$) \cite{Dys62,Kat16_Springer}. The SDEs (\ref{eqn:SDE_C3}) and (\ref{eqn:SDE_D3}) are known as the system of {\it noncolliding absorbing Brownian motions} (or the Dyson model of type C) and the system of {\it noncolliding reflecting Brownian motions} (or the Dyson model of type D), respectively. See \cite{BFPSW09,KT04,KT07a,KT11,KO01,LW16,TW07}.

The correlation kernel for a determinantal process is in general a function of two points on the spatio-temporal plane, say, $(s, x)$ and $(t, y)$. In the present formula (\ref{eqn:K1}) in Theorem \ref{thm:main1}, the dependence of the correlation kernel $\K^{\RN}_{\u}$ on one spatio-temporal point $(s, x)$ is explicitly given by the transition probability density $p^{[0, \pi r]}$ of a single Brownian motion in an interval $[0, \pi r]$ with given boundary conditions. The dependence of $\K^{\RN}_{\u}$ on another spatio-temporal point $(t, y)$ is described by the following two procedures. The information of interaction among particles and initial conf\/iguration is expressed by a set of entire functions $\big\{\Phi^{\RN}_{\u, u_j}(z) \big\}_{j=1}^N$, (\ref{eqn:Phi_Ra})--(\ref{eqn:Phi_Rc}), which are {\it static} functions without time-variable $t$. Then evolution in time~$t$ is given by the integral transformation (\ref{eqn:MB1}) specif\/ied by the transition probability density $p_{\rm BM}$ with time duration~$t$ of a~single Brownian motion. One of the benef\/its of such separation of static information and dynamics in the formula is that we can trace {\it a relaxation process to equilibrium}, which is a~typical non-equilibrium phenomenon, for any initial conf\/iguration in $\W_N^{(0, \pi r)}$~\cite{Kat14}. In order to demonstrate this fact, in Section~\ref{sec:relax} we will take the temporally homogeneous limit~$t_{\ast} \to \infty$ to make the systems have equilibrium processes and study the relaxation to equilibria for the trigonometric Dysom models of types C and D. Another possible benef\/it of the present formula for spatio-temporal correlation kernels is that the inf\/inite particle systems will be studied if we can control the set of entire functions~$\big\{\Phi^{\RN}_{\u, u_j}(z) \big\}_{j=1}^N$ in the inf\/inite-particle limit $N \to \infty$. For the original Dyson model~\cite{KT10} and other related systems~\cite{KT09,KT11}, we have applied the Hadamard theorem on the Weierstrass canonical product-formulas of entire functions~\cite{Lev96} to analyze the inf\/inite particle systems. In the present case with (\ref{eqn:Phi_Ra})--(\ref{eqn:Phi_Rc}), we have to treat the rations of inf\/inite products of the Jacobi theta functions in this limit $N \to \infty$. Asymptotic analysis with $N \to \infty$ will be an interesting future problem.

The paper is organized as follows. In Section \ref{sec:Det_Eq} we explain the relationship between determinantal martingales and determinantal processes used in this paper, and determinantal equalities (Lemma \ref{thm:RS}) obtained from the elliptic determinant evaluations of the Macdonald denominators given by Rosengren and Schlosser \cite{RS06}. There we explain how to derive the determinantal martingale-functions (\ref{eqn:D_Ra})--(\ref{eqn:D_Rc}), which def\/ine the interacting systems of Brownian motions by~(\ref{eqn:Plaw_R1}). Proofs of Theorems \ref{thm:main1} and \ref{thm:main2} are given in Sections~\ref{sec:proof1} and~\ref{sec:proof2}, respectively. We study the temporally homogeneous limit $t_{\ast} \to \infty$ in Section~\ref{sec:relax}, and relaxation processes to equilibria are clarif\/ied for the trigonometric Dyson models of types C and D. Concluding remarks and open problems are given in Section~\ref{sec:remarks}. Notations and formulas of the Jacobi theta functions and related functions are listed in Appendix~\ref{sec:appendixA}.

\section{Determinantal martingales and determinantal equalities} \label{sec:Det_Eq}
\subsection{Notion of martingale} \label{sec:notion_martingale}

Martingales are the stochastic processes preserving their mean values and thus they represent f\/luctuations. A typical example of martingale is the one-dimensional standard Brownian motion as explained below. Let $B(t)$, $t \geq 0$ denote the position of the standard Brownian motion in $\R$ starting from the origin 0 at time $t=0$. The transition probability density from the position $x \in \R$ to $y \in \R$ in time duration $t \geq 0$ is given by~(\ref{eqn:p_BM}). Let $\cB(\R)$ be the Borel set on $\R$. Then for an arbitrary time sequence $0 \equiv t_0 < t_1 < \cdots < t_M < \infty$, $M \in \N \equiv \{1,2, \dots \}$, and for any $A_m \in \cB(\R)$, $m=1,2, \dots, M$,
\begin{gather*}
\rP\big[B(t_m) \in A_m, \, m=1,2, \dots, M \big]\\
\qquad{} =\int_{A_1} d x^{(1)} \cdots \int_{A_M} d x^{(M)}\prod_{m=1}^M p_{\rm BM}\big(t_m-t_{m-1}, x^{(m)} | x^{(m-1)}\big)
\end{gather*}
with $x^{(0)} \equiv 0$. The collection of all paths is denoted by $\Omega$ and there is a subset $\widetilde{\Omega} \subset \Omega$ such that $\rP[\widetilde{\Omega}]=1$ and for any realization of path $\omega \in \widetilde{\Omega}$, $B(t)=B(t, \omega)$, $t \geq 0$ is a real continuous function of~$t$. In other words, $B(t)$, $t \geq 0$ has a continuous path almost surely (a.s.\ for short). For each $t \in [0, \infty)$, we write the smallest
$\sigma$-f\/ield (completely additive class of events) generated by the Brownian motion up to time $t$ as $\cF_t=\sigma(B(s)\colon 0 \leq s \leq t)$. We have a nondecreasing family $\{\cF_t \colon t \geq 0 \}$ such that $\cF_s \subset \cF_t$ for $0 \leq s < t < \infty$, which we call a {\it filtration}, and put $\cF=\bigcup_{t \geq 0} \cF_t$. The triplet $(\Omega, \cF, \rP)$ is called the probability space for the one-dimensional standard Brownian motion, and $\rP$ is especially called the Wiener measure. The expectation with respect to the probability law $\rP$ is written as $\rE$. When we see $p_{\rm BM}(t, y|x)$ as a function of $y$, it is nothing but the probability density of the normal distribution with mean~$x$ and variance~$t$, and hence it is easy to verify that
\begin{gather*}
\rE[B(t)| \cF_s] = \int_{-\infty}^{\infty} x p_{\rm BM}(t-s, x|B(s)) d x =B(s) \qquad \text{a.s.} \quad 0 \leq s < t < \infty,
\end{gather*}
which means that $B(t)$, $t \geq 0$ is a martingale. We see, however, $B(t)^n$, $t \geq 0$, $n \in \{2,3, \dots\}$ are not martingales, since the generating function of $\rE[B(t)^n | \cF_s]$, $n \in \N_0 \equiv \{0, 1, 2, \dots\}$, $0< s \leq t < \infty$, with parameter $\alpha \in \C$ is calculated as
\begin{gather*}
\sum_{n=0}^{\infty} \frac{\alpha^n}{n!} \rE[ B(t)^n | \cF_s]= \rE[ e^{\alpha B(t)} | \cF_s]
= \int_{-\infty}^{\infty} e^{\alpha x} \frac{e^{-(x-B(s))^2/2(t-s)}}{\sqrt{2 \pi (t-s)}} d x \\
\hphantom{\sum_{n=0}^{\infty} \frac{\alpha^n}{n!} \rE[ B(t)^n | \cF_s]= \rE[ e^{\alpha B(t)} | \cF_s]}{} = e^{\alpha B(s)+(t-s) \alpha^2/2}
\not= e^{\alpha B(s)} \quad \text{a.s.}
\end{gather*}
Now we assume that $\widetilde{B}(t)$, $t \geq 0$ is a one-dimensional standard Brownian motion which is independent of $B(t)$, $t \geq 0$, and its probability space is denoted by $\big(\widetilde{\Omega}, \widetilde{\cF}, \widetilde{\rP}\big)$. Then we introduce a complex-valued martingale called the {\it complex Brownian motion}
\begin{gather*}
Z(t)=B(t)+ i \widetilde{B}(t), \qquad t \geq 0,
\end{gather*}
with $i=\sqrt{-1}$. The probability space of $Z(t)$, $t \geq 0$ is given by the product space $(\Omega, \cF, \rP) \otimes (\widetilde{\Omega}, \widetilde{\cF}, \widetilde{\rP})$ and we write the expectation as $\bE=\rE \otimes \widetilde{\rE}$. For the complex Brownian motion, by the independence of its real and imaginary parts, we see that
\begin{gather*}
 \sum_{n=0}^{\infty} \frac{\alpha^n}{n!}\bE\big[ Z(t)^n | \cF_s \otimes \widetilde{\cF}_s \big]= \bE\big[e^{\alpha Z(t)} | \cF_s \otimes \widetilde{\cF}_s \big]= \rE\big[e^{\alpha B(t)} | \cF_s\big]\times \widetilde{\rE}\big[e^{i \alpha \widetilde{B}(t)} | \widetilde{\cF}_s\big] \\
\hphantom{\sum_{n=0}^{\infty} \frac{\alpha^n}{n!}\bE\big[ Z(t)^n | \cF_s \otimes \widetilde{\cF}_s \big]}{}
= e^{\alpha B(s) + (t-s) \alpha^2/2} \times e^{i \alpha \widetilde{B}(s) -(t-s) \alpha^2/2} = e^{\alpha Z(s)} = \sum_{n=0}^{\infty} \frac{\alpha^n}{n!} Z(s)^n\\
\hphantom{\sum_{n=0}^{\infty} \frac{\alpha^n}{n!}\bE\big[ Z(t)^n | \cF_s \otimes \widetilde{\cF}_s \big]=}{}
 \text{a.s.} \quad 0 \leq s < t < \infty, \quad \alpha \in \C,
\end{gather*}
which implies that for any $n \in \N_0$, $Z(t)^n$, $t \geq 0$ is a martingale. This observation will be generalized as the following stronger statement; if $F$ is an entire and non-constant function, then $F(Z(t))$, $t \geq 0$ is a time change of a complex Brownian motion (see, for instance, \cite[Theorem~2.5 in Section~V.2]{RY99}). Hence $F(Z(t))$, $t \geq 0$ is a martingale:
\begin{gather*}
\bE[ F(Z(t)) | \cF_s \otimes \widetilde{\cF}_s ] = F(Z(s)) \qquad \mbox{a.s.} \quad 0 \leq s < t < \infty.
\end{gather*}
If we take the expectation $\widetilde{\rE}$ with respect to $\Im Z(\cdot)=\widetilde{B}(\cdot)$ of the both sides of the above equality, we have
\begin{gather*}
\rE \big[ \widetilde{\rE}[F(Z(t))] \big| \cF_s \big]
=\widetilde{\rE}[F(Z(s))] \qquad \text{a.s.} \quad 0 \leq s < t < \infty.
\end{gather*}
In this way we can obtain a martingale $\widehat{F}(t, B(t)) \equiv \widetilde{\rE}[F(Z(t))]$, $t \geq 0$ with respect to the f\/iltration~$\cF_t$ of the one-dimensional Brownian motion. The present argument implies that if we have proper entire functions, then we will obtain useful martingales describing intrinsic f\/luctuations involved in interacting particle systems.

\subsection{Basic equalities and determinantal martingales} \label{sec:basic}

Let $f_j$, $j \in \I$, be an inf\/inite series of linearly independent entire functions, where the index set $\I$ is $\Z =\{\dots, -1, 0, 1, 2, \dots\}$ or~$\N$. For $N \in \{2,3 \dots\}$, we assume
\begin{gather}
\u=(u_1, u_2, \dots, u_N) \in \W_N.\label{eqn:u1}
\end{gather}
We then def\/ine a set of $N$ distinct entire and non-constant functions of $z \in \C$ by
\begin{gather}
\Phi_{\u, u_j}(z)=\frac{{\det\limits_{1 \leq \ell, m \leq N}[ f_{\ell}(u_m+(z-u_m) \delta_{mj})]}}
{{\det\limits_{1 \leq \ell, m \leq N}[ f_{\ell}(u_m) ]}},\qquad j =1,2, \dots, N.\label{eqn:Phi1}
\end{gather}
By def\/inition, it is easy to verify that
\begin{gather}
\Phi_{\u, u_j}(u_k)=\delta_{jk}, \qquad j, k \in \{1, 2, \dots, N\}. \label{eqn:interpolation1}
\end{gather}
It should be noted that given (\ref{eqn:u1}), $\Phi_{\u, u_j}(z)$, $j \in \{1,2, \dots, N\}$ satisfying (\ref{eqn:interpolation1}) can be regarded as the {\it Lagrange interpolation functions} \cite{IN16,IN17a,IN17b}. We can prove the following lemmas.

\begin{Lemma}\label{thm:det_lemma_1}Let $N \in \{2,3, \dots\}$. The functions \eqref{eqn:Phi1} have the following expansions
\begin{gather}
\Phi_{\u, u_j}(z)=\sum_{k=1}^N \phi_{\u, u_j}(k) f_k(z), \qquad j \in \{1,2, \dots, N\},\label{eqn:PhiB1}
\end{gather}
and the coefficients $\phi_{\u, u_j}(k)$, $j, k \in \{1,2, \dots, N\}$ satisfy the relations
\begin{gather}
\sum_{\ell=1}^N f_j(u_{\ell}) \phi_{\u, u_{\ell}}(k)= \frac{{\det\limits_{1 \leq m, n \leq N} [ f_{m+(j-m) \delta_{m k}}(u_n) ]}}
{{\det\limits_{1 \leq m, n \leq N} [ f_m(u_n) ]}}, \qquad j \in \Z, \qquad k \in \{1,2, \dots, N\}. \label{eqn:fPhi1}
\end{gather}
In particular, if $j$ is an element of $\{1,2, \dots, N\}$,
\begin{gather}
\sum_{\ell=1}^N f_j(u_{\ell}) \phi_{\u, u_{\ell}}(k) = \delta_{jk}, \qquad j, k \in \{1,2, \dots, N\}.\label{eqn:fPhi2}
\end{gather}
\end{Lemma}

\begin{proof}Consider an $N \times N$ matrix
\begin{gather*}
\sff_{\u}=(f_{\ell}(u_{m}))_{1 \leq \ell, m \leq N},
\end{gather*}
and write its determinant as $|\sff_{\u}| \not=0$. The minor determinant $|\sff_{\u}(k, j)|$ is def\/ined as the determinant of the $(N-1)\times(N-1)$ matrix, $\sff_{\u}(k, j)$, which is obtained from $\sff_{\u}$ by deleting the $k$-th row and the $j$-th column, $k, j \in \{1,2, \dots, N\}$. Then the determinant in the numerator of (\ref{eqn:Phi1}) is expanded along the $j$-th column and we obtain~(\ref{eqn:PhiB1}) with
\begin{gather}
\phi_{\u, u_j}(k)=(-1)^{j+k} \frac{|\sff_{\u}(k, j)|}{|\sff_{\u}|}.\label{eqn:phi1}
\end{gather}
Hence
\begin{gather*}
\sum_{\ell=1}^N f_j(u_{\ell}) \phi_{\u, u_{\ell}}(k)=\frac{1}{|\sff_{\u}|} \sum_{\ell=1}^N (-1)^{\ell+k} |\sff_{\u}(k, \ell)| f_j(u_{\ell}),
\end{gather*}
which proves (\ref{eqn:fPhi1}), for the summation in the r.h.s.\ is the expansion of $\det\limits_{1 \leq m,n \leq N}[f_{m+(j-m)\delta_{mk}}(u_n)]$ along the $k$-th row. It is immediate to conclude (\ref{eqn:fPhi2}) from (\ref{eqn:fPhi1}), if both of $j$ and $k$ are elements of $\{1,2, \dots, N\}$.
\end{proof}

Then we obtain the following equalities.

\begin{Lemma}\label{thm:det_lemma_2} \quad
\begin{enumerate}\itemsep=0pt
\item[$(i)$] For $z \in \C$,
\begin{gather}
\sum_{\ell=1}^N f_j(u_{\ell}) \Phi_{\u, u_{\ell}}(z) =f_j(z), \qquad j \in \{1,2, \dots, N\}. \label{eqn:fPhi3}
\end{gather}
\item[$(ii)$] For $\z=(z_1, z_2, \dots, z_N) \in \C^N$, the equality
\begin{gather}
\det_{1 \leq j, k \leq N} [ \Phi_{\u, u_j}(z_k) ]=\frac{{\det\limits_{1 \leq j, k \leq N} [ f_j(z_k) ]}}
{{\det\limits_{1 \leq j, k \leq N} [ f_j(u_k) ]}}, \qquad \z=(z_1, z_2, \dots, z_N) \in \C^N \label{eqn:Phi1c}
\end{gather}
holds.
\end{enumerate}
\end{Lemma}

\begin{proof}(i) Multiply the both sides of (\ref{eqn:fPhi2}) by $f_k(z)$ and take summation over $k \in \{1,2, \dots, N\}$. Then use~(\ref{eqn:PhiB1}) to obtain (\ref{eqn:fPhi3}). (ii) Consider $N$ equations by putting $z=z_j$, $j \in \{1,2, \dots, N\}$ in~(\ref{eqn:fPhi3}). Then calculate the determinant $|\sff_{\z}|$ to prove the determinantal equality~(\ref{eqn:Phi1c}).
\end{proof}

We consider $N$ pairs of independent copies $(B_k(t), \widetilde{B}_k(t))$, $k=1,2, \dots, N$, of $(B(t), \widetilde{B}(t))$, $t \geq 0$, and def\/ine~$N$ independent complex Brownian motions
\begin{gather}
Z_k(t)=u_k+B_k(t)+ i \widetilde{B}_k(t), \qquad t \geq 0, \qquad k=1,2, \dots, N,\label{eqn:Z1}
\end{gather}
each of which starts from $u_k \in \R$. The probability law and expectation of them are given by $\rP_{\u} \otimes \widetilde{\rP}$ and $\rE_{\u} \otimes \widetilde{\rE}$, respectively. Then for each complex Brownian motion $Z_k(t)$, $t \geq 0$, $k=1,2, \dots, N$, we have $N$ distinct time-changes of complex Brownian motions, $\Phi_{\u, u_j}(Z_k(t))$, $t \geq 0$, $j=1,2, \dots, N$, started from the real values, 0 or 1; $\Phi_{\u, u_j}(Z_k(0))=\Phi_{\u, u_j}(u_k)=\delta_{jk}$, $j, k=1,2, \dots, N$. Therefore, we can conclude that, if we take expectation $\widetilde{\rE}$, we will obtain $N$ distinct martingales
\begin{gather*}
M_{\u, u_j}(t, B_k(t))\equiv \widetilde{\rE}\big[ \Phi_{\u, u_j}(B_k(t)+i \widetilde{B}_k(t))\big]\nonumber\\
\hphantom{M_{\u, u_j}(t, B_k(t))}{}
= \frac{{\det\limits_{1 \leq \ell, m \leq N}\big[\widehat{f}_{\ell}(t \delta_{m j}, u_m+(B_k(t)-u_m) \delta_{m j})\big]}}
{{\det\limits_{1 \leq \ell, m \leq N}\big[\widehat{f}_{\ell}(0, u_m)\big]}},\qquad t \geq 0, \quad j=1,2, \dots, N,\nonumber
\end{gather*}
for each one-dimensional Brownian motion $B_k(t)$, $t \geq 0$, $k=1,2, \dots, N$, where
\begin{gather*}
\widehat{f}_{\ell}(t, x) \equiv\widetilde{\rE}\big[ f_{\ell}(x+i \widetilde{B}(t))\big]= \int_{-\infty}^{\infty} f_{\ell}(x+ i \widetilde{x}) p_{\rm BM}(t, \widetilde{x}|0) d \widetilde{x}, \qquad \ell \in \I.
\end{gather*}
By this def\/inition, $\widehat{f}_{\ell}(0, x)=f_{\ell}(x)$, $\ell \in \I$. Then we def\/ine a multivariate function of $t \geq 0$ and $\x=(x_1, x_2, \dots, x_N) \in \R^N$ by
\begin{gather}
D_{\u}(t, \x)=\det_{1 \leq j, k \leq N} [ M_{\u, u_j}(t, x_k)], \label{eqn:D1}
\end{gather}
which gives a martingale as a functional of $t \geq 0$ and $\B(t)=(B_1(t), B_2(t), \dots, B_N(t))$, $t \geq 0$,
\begin{gather*}
\rE_{\u} [ D_{\u}(t, \B(t)) | \cF_s] =D_{\u}(s, \B(s)), \qquad 0 \leq s < t < \infty.
\end{gather*}
By the multi-linearity of determinant and (\ref{eqn:Phi1c}), we see
\begin{gather}
D_{\u}(t, \x) = \widetilde{\rE}\Big[ \det_{1 \leq j, k \leq N}\big[ \Phi_{\u, u_j}(x_k+i \widetilde{B}_k(t))\big] \Big]
= \frac{{\det\limits_{1 \leq j, k \leq N} \big[ \widehat{f}_j(t, x_k) \big]}}
{{\det\limits_{1 \leq j, k \leq N} \big[ \widehat{f}_j(0, u_k)\big]}}. \label{eqn:D3}
\end{gather}
We call $D_{\u}(t, \B(t))$, $t \geq 0$, the {\it determinantal martingale} \cite{Kat14}.

\subsection{Auxiliary measure and spatio-temporal Fredholm determinant} \label{sec:auxiliary}

Now we introduce an auxiliary measure $\widehat{\P}_{\u}$, which is complex-valued in general, but is absolutely continuous to the Wiener measure $\rP_{\u}$ as
\begin{gather}
\widehat{\P}_{\u} \big|_{\cF_t} = D_{\u}(t, \B(t)) \rP_{\u} \big|_{\cF_t}, \qquad t \geq 0.\label{eqn:Phat1}
\end{gather}
(Note that this is generally dif\/ferent from $\P^{\RN}_{\u}$ given by (\ref{eqn:Plaw_R1}), since the condition ${\bf 1}(T_{\rm collision}>t)$ is omitted, and hence it is {\it not} the measure representing any noncolliding particles.) Here we consider the corresponding auxiliary system of~$N$ particles on $\R$ governed by $\widehat{\P}_{\u}$
\begin{gather*}
\widehat{\X}(t)=\big(\widehat{X}_1(t), \dots, \widehat{X}_N(t)\big), \qquad t \geq 0,
\end{gather*}
starting from (\ref{eqn:u1}) and each particle of which has a continuous path a.s. We consider the unlabeled conf\/iguration of $\widehat{\X}(t)$ as
\begin{gather}
\widehat{\Xi}(t, \cdot)=\sum_{j=1}^N \delta_{\widehat{X}_j(t)}(\cdot), \qquad t \geq 0. \label{eqn:hatX}
\end{gather}
Consider an arbitrary number $M \in \N$ and an arbitrary set of strictly increasing times $\t=\{t_1, t_2, \dots, t_M\}$, $0 \equiv t_0 < t_1 < \cdots < t_M < \infty$. Let $\rC_{\rc}(\R)$ be the set of all continuous real-valued functions with compact supports on~$\R$. For $\mbg=(g_{t_1}, g_{t_1}, \dots, g_{t_M}) \in \rC_{\rc}(\R)^M$, we consider the following functional of $\mbg$
\begin{gather}
\widehat{\L}_{\u}[\t, \mbg]= \widehat{\E}_{\u} \left[ \exp \left\{ \sum_{m=1}^{M}\int_{\R} g_{t_m}(x) \widehat{\Xi}(t_m, dx) \right\} \right], \label{eqn:GF1}
\end{gather}
which is the Laplace transform of the multi-time distribution function $\widehat{\P}_{\u}$ on a set of times $\t$ with the functions $\mbg$. If we put
\begin{gather*}
\chi_t=e^{g_t}-1,
\end{gather*}
then (\ref{eqn:GF1}) can be written as
\begin{gather}
\widehat{\L}_{\u}[\t, \mbg]=\widehat{\E}_{\u} \left[\prod_{m=1}^M \prod_{j=1}^N \big\{1+ \chi_{t_m}\big(\widehat{X}_j(t_m)\big)\big\}\right].\label{eqn:GF2}
\end{gather}
Explicit expression of (\ref{eqn:GF2}) is given by the following multiple integrals
\begin{gather*}
\widehat{\L}_{\u}[\t, \mbg]=\int_{\R^N} d \x^{(1)} \cdots \int_{\R^N} d \x^{(M)} D_{\u}\big(t_M, \x^{(M)}\big) \\
\hphantom{\widehat{\L}_{\u}[\t, \mbg]=}{} \times \prod_{m=1}^M \prod_{j=1}^N \big[ p_{\rm BM}\big(t_m-t_{m-1}, x^{(m)}_j|x^{(m-1)}_j\big)
\big\{ 1+ \chi_{t_m}\big(x^{(m)}_j\big) \big\} \big],
\end{gather*}
where $x^{(0)}_j=u_j$, $j=1,2, \dots, N$, $\x^{(m)}=\big(x^{(m)}_1,\dots,x^{(m)}_N\big)$, and $d \x^{(m)}=\prod\limits_{j=1}^N dx^{(m)}_j$, $m=1,2,\dots,M$.

The following was proved as Theorem 1.3 in~\cite{Kat14}.

\begin{Proposition}\label{thm:Fredholm}Put
\begin{gather*}
\widehat{\K}_{\u}(s, x; t, y)=\sum_{j=1}^N p_{\rm BM}(s, x|u_j) M_{\u, u_j}(t, y)- {\bf 1}(s>t) p_{\rm BM}(s-t, x|y),
\end{gather*}
$s, t > 0$, $x, y \in \R$. Then the following equality holds for an arbitrary number $M \in \N$, an arbitrary set of strictly increasing times $\t=\{t_1, t_2, \dots, t_M\}$, $0 \equiv t_0 < t_1 < \cdots < t_M < \infty$, and $\mbg=(g_{t_1}, g_{t_1}, \dots, g_{t_M}) \in \rC_{\rc}(\R)^M$,
\begin{gather*}
\widehat{\L}_{\u}[\t, \mbg]=\mathop{{\rm Det}}_{\substack{s, t \in \t, \\ x, y \in \R}}
\big[\delta_{st} \delta(x-y)+ \widehat{\K}_{\u}(s, x; t, y) \chi_{t}(y) \big],
\end{gather*}
where the r.h.s.\ denotes the spatio-temporal Fredholm determinant with the kernel $\widehat{\K}_{\u}$ defined by
\begin{gather}
 \mathop{{\rm Det}}_{\substack{s, t \in \t, \\ x, y \in \R}}
\big[\delta_{st} \delta(x-y)+ \widehat{\K}_{\u}(s, x; t, y) \chi_{t}(y) \big]\label{eqn:Fredholm}\\
\equiv \sum_{\substack{0 \leq N_m \leq N, \\ 1 \leq m \leq M} }
\int_{\prod\limits_{m=1}^{M} \W_{N_{m}}}\prod_{m=1}^{M}d \x_{N_m}^{(m)}\prod_{j=1}^{N_{m}}\chi_{t_m} \big(x_{j}^{(m)}\big)
\det_{\substack{1 \leq j \leq N_{m},\, 1 \leq k \leq N_{n}, \\
1 \leq m, n \leq M}}\big[\widehat{\K}_{\u}\big(t_m, x_{j}^{(m)}; t_n, x_{k}^{(n)}\big)\big],\nonumber
\end{gather}
where $d \x^{(m)}_{N_m}= \prod\limits_{j=1}^{N_m} dx^{(m)}_j$, $m=1,2, \dots, M$, and the term with $N_m=0$, $1 \leq \forall\, m \leq M$ in the r.h.s.\ should be interpreted as~$1$.
\end{Proposition}

This proposition is general and it proves that the auxiliary system (\ref{eqn:hatX}) is determinantal. The measure (\ref{eqn:Phat1}) which governs this particle system
is, however, complex-valued in general, and hence the system is unphysical. The problem is to clarify the proper conditions which should be added to (\ref{eqn:Phat1}) to construct a non-negative-def\/inite real measure, i.e., the probability measure, which def\/ines a physical system of interacting particles. As a matter of course, this problem depends on the choice of an inf\/inite set of linearly independent entire functions~$f_j$, $j \in \I$.

\subsection{Elliptic determinant evaluations of the Macdonald denominators} \label{sec:Macdonald}

Here we report the results when we choose the entire functions $f_j$, $j \in \Z$ as follows
\begin{gather}
f^{\AN}_j(z; \tau) = e^{i J^{\AN}(j) z/r} \vartheta_1 \left(
J^{\AN}(j) \tau + \frac{\cN^{\AN} z}{2 \pi r}+\frac{1-(-1)^N}{4}; \cN^{\AN} \tau \right),\nonumber\\
f^{\RN}_j(z; \tau) = e^{i J^{\RN}(j) z/r} \vartheta_1 \left(J^{\RN}(j) \tau + \frac{\cN^{\RN} z}{2 \pi r}; \cN^{\RN} \tau \right)\nonumber\\
\hphantom{f^{\RN}_j(z; \tau) =}{} -
e^{-i J^{\RN}(j) z/r} \vartheta_1 \left(J^{\RN}(j) \tau - \frac{\cN^{\RN} z}{2 \pi r}; \cN^{\RN} \tau \right)\qquad\text{for $\RN=\BN, \BNv$},\nonumber\\
f^{\RN}_j(z; \tau) = e^{i J^{\RN}(j) z/r} \vartheta_1 \left(J^{\RN}(j) \tau + \frac{\cN^{\RN} z}{2 \pi r} + \frac{1}{2} ; \cN^{\RN} \tau \right)\nonumber\\
\hphantom{f^{\RN}_j(z; \tau) =}{} -
e^{-i J^{\RN}(j) z/r} \vartheta_1 \left(
J^{\RN}(j) \tau - \frac{\cN^{\RN} z}{2 \pi r} + \frac{1}{2} ; \cN^{\RN} \tau \right)\qquad \text{for $\RN=\CN, \CNv, \BCN$},\nonumber\\
f^{\DN}_j(z; \tau) = e^{i J^{\DN}(j) z/r} \vartheta_1 \left(J^{\DN}(j) \tau + \frac{\cN^{\DN}z}{2 \pi r}+\frac{1}{2}; \cN^{\DN} \tau \right)\nonumber\\
\hphantom{f^{\DN}_j(z; \tau) =}{} +
e^{-i J^{\DN}(j) z/r} \vartheta_1 \left(J^{\DN}(j) \tau - \frac{\cN^{\DN}z}{2 \pi r} +\frac{1}{2}; \cN^{\DN} \tau \right),\label{eqn:f}
\end{gather}
$z \in \C$, $j \in \Z$, with $N \in \N$, $0 < r < \infty$, and $\tau \in \C$ with $0 < \Im \tau < \infty$,
where
\begin{gather}
J^{\RN}(j) = \begin{cases}
j-1, & \RN=\AN, \BN, \BNv, \DN,\\
j, & \RN=\CN, \BCN,\\
j-1/2,& \RN = \CNv,
\end{cases} \label{eqn:J_R}
\end{gather}
and $\cN^{\RN}$ are given by (\ref{eqn:cN_R}). These functions (\ref{eqn:f}) were used to express the determinant evaluations by Rosengren and Schlosser~\cite{RS06} for the Macdonald denominators $W_{\RN}(x)$ for seven families of irreducible reduced af\/f\/ine root systems $\RN=\AN, \BN, \BNv, \CN, \CNv, \BCN, \DN$. Note that, if $x \in \R$, $\tau \in i \R$, and $0 < \Im \tau < \infty$, then $f^{\RN}(x; \tau) \in \R$ for $\RN=\BN, \BNv, \DN$ and $f^{\RN}(x; \tau) \in i \R$ for $\RN=\CN, \CNv, \BCN$.

Assume that, for $\u=(u_1, u_2, \dots, u_N) \in \W_N$, $\det\limits_{1 \leq j, k \leq N}[f_j(u_k)]$ is factorized in the form
\begin{gather}
\det_{1 \leq j, k \leq N} [f_j(u_k)]= k_0 k_{\rm sym}(\u) \prod_{\ell=1}^N k_1(u_{\ell})\prod_{1 \leq j < k \leq N} k_2(u_k, u_j),\label{eqn:factor1}
\end{gather}
where $k_0$ is a constant, $k_1$ is a single variable function, $k_2$ is an antisymmetric function of two variables, $k_2(u, v)=-k_2(v, u)$, and $k_{\rm sym}(\u)$ is a symmetric function of $\u$ which cannot be factorized as $\prod\limits_{\ell=1}^N k_1(u_{\ell})$. For $\u=(u_1, \dots, u_{j-1}, u_j, u_{j+1}, \dots, u_N) \in \W_N$, we replace the $j$-th component $u_j$ by a variable $z$ and write the obtained vector as $\u^{(j)}(z)=(u_1, \dots, u_{j-1}, z, u_{j+1}, \dots, u_N)$. Under the assumption (\ref{eqn:factor1}), $\Phi_{\u, u_j}(z)$ def\/ined by (\ref{eqn:Phi1}) is also factorized as
\begin{gather}
\Phi_{\u, u_j}(z)= \frac{k_{\rm sym}( \u^{(j)}(z))}{k_{\rm sym}(\u)}\frac{k_1(z)}{k_1(u_j)}
\prod_{\substack{1 \leq \ell \leq N, \cr \ell \not=j}}\frac{k_2(z, u_{\ell})}{k_2(u_j, u_{\ell})},\qquad j=1,2, \dots, N.\label{eqn:PhiB2}
\end{gather}
Then the determinantal equality (\ref{eqn:Phi1c}) becomes
\begin{gather}
\det_{1 \leq j, k \leq N}\left[ \frac{k_{\rm sym}(\u^{(j)}(z_k))}{k_{\rm sym}(\u)}\frac{k_1(z_k)}{k_1(u_j)}
\prod_{\substack{1 \leq \ell \leq N, \cr \ell \not=j}}\frac{k_2(z_k, u_{\ell})}{k_2(u_j, u_{\ell})} \right]\nonumber\\
\qquad{} = \frac{k_{\rm sym}(\z)}{k_{\rm sym}(\u)}\prod_{\ell=1}^N \frac{k_1(z_{\ell})}{k_1(u_{\ell})}\prod_{1 \leq j < k \leq N}\frac{k_2(z_k, z_j)}{k_2(u_k, u_j)},\qquad \z \in \C^N, \qquad \u \in \W_N.\label{eqn:PhiB3}
\end{gather}

In \cite{RS06} Rosengren and Schlosser gave the elliptic determinant evaluations of the Macdonald denominators of seven families of irreducible reduced af\/f\/ine root systems, $\AN$, $\BN$, $\BNv$, $\CN$, $\CNv$, $\BCN$ and $\DN$. From their determinantal equalities \cite[Proposition~6.1]{RS06}, it is easy to verify that our seven choices of $f_j$, $j \in \Z$ given by~(\ref{eqn:f}) allow the factorization~(\ref{eqn:factor1}) for their determinants. As functions of $\tau$, we write
\begin{gather}
q(\tau)=e^{\tau \pi i}, \qquad q_0(\tau)= \prod_{n=1}^{\infty} \big(1-q(\tau)^{2n} \big).\label{eqn:q_q0}
\end{gather}

\begin{Lemma}\label{thm:RS}For the seven choices of $f^{\RN}_j(\cdot; \tau)$, $j \in \Z$, $\Im \tau > 0$, given by \eqref{eqn:f}, the equality
\begin{gather}
\det_{1 \leq j, k \leq N} \big[f^{\RN}_j(u_k; \tau)\big]= k^{\RN}_0(\tau) k^{\RN}_{\rm sym}(\u; \tau) \prod_{\ell=1}^N k^{\RN}_1(u_{\ell}; \tau)
\prod_{1 \leq j < k \leq N} k^{\RN}_2(u_k, u_j; \tau),\label{eqn:factor1b}
\end{gather}
holds with the following factors,
\begin{gather}
k_0^{\AN}(\tau) = i^{-(N-1)(3 N+1-(-1)^N)/2} q(\tau)^{-(N-1)(3N-2)/8} q_0(\tau)^{-(N-1)(N-2)/2},\nonumber\\
k_0^{\BN}(\tau) = 2 q(\tau)^{-N(N-1)/4} q_0(\tau)^{-N(N-1)},\nonumber\\
k_0^{\BNv}(\tau) = 2 q(\tau)^{-N(N-1)/4} q_0(\tau)^{-(N-1)^2} q_0(2 \tau)^{-(N-1)},\nonumber\\
k_0^{\CN}(\tau) = i^{-N} q(\tau)^{-N^2/4} q_0(\tau)^{-N(N-1)},\nonumber\\
k_0^{\CNv}(\tau) = i^{-N} q(\tau)^{-N(2N-1)/8} q_0(\tau)^{-(N-1)^2} q_0(\tau/2)^{-(N-1)},\nonumber\\
k_0^{\BCN}(\tau) = i^{-N} q(\tau)^{-N(N+1)/4} q_0(\tau)^{-N(N-1)} q_0(2 \tau)^{-N},\nonumber\\
k_0^{\DN}(\tau) = 4 q(\tau)^{-N(N-1)/4} q_0(\tau)^{-N(N-2)},\label{eqn:k0}\\
k_{\rm sym}^{\AN}(\u; \tau)=\vartheta_1 \left( \frac{\sum\limits_{j=1}^N u_j -\kappa_N}{2 \pi r}; \tau \right),\nonumber\\
k_{\rm sym}^{\RN}(\u; \tau)=1 \quad \mbox{for $\RN =\BN, \BNv, \CN, \CNv, \BCN, \DN$},\nonumber\\
k_1^{\AN}(u; \tau)=k_1^{\DN}(u; \tau)=1,\nonumber\\
k_1^{\BN}(u; \tau)=\vartheta_1 \left( \frac{u}{2 \pi r}; \tau \right),\qquad k_1^{\BNv}(u; \tau)=\vartheta_1 \left( \frac{u}{\pi r}; 2 \tau \right),
\nonumber\\
k_1^{\CN}(u; \tau)=\vartheta_1 \left( \frac{u}{\pi r}; \tau \right), \qquad k_1^{\CNv}(u; \tau)=\vartheta_1 \left( \frac{u}{2 \pi r}; \frac{\tau}{2} \right),
\nonumber\\
k_1^{\BCN}(u; \tau)=\vartheta_1 \left( \frac{u}{2 \pi r}; \tau \right)\vartheta_0 \left( \frac{u}{\pi r}; 2 \tau \right),\nonumber\\
k_2^{\AN}(u, v; \tau)=\vartheta_1 \left( \frac{u-v}{2\pi r}; \tau \right),\nonumber\\
k_2^{\RN}(u, v; \tau)=\vartheta_1 \left( \frac{u-v}{2\pi r}; \tau \right)\vartheta_1 \left( \frac{u+v}{2\pi r}; \tau \right)
\quad \text{for $\RN =\BN, \BNv, \CN, \CNv, \BCN, \DN$}.\label{eqn:factor2}
\end{gather}
\end{Lemma}

\subsection{Determinantal martingale-functions} \label{sec:martingale}

Since $f_j^{\RN}(z; \tau)$, $j \in \Z$ given by (\ref{eqn:f}) are entire and non-constant functions of $z \in \C$,
\begin{gather}
\widehat{f}^{\RN}_j(t, x; \tau) =\widetilde{\rE} [ f^{\RN}_j(x+ i \widetilde{B}(t); \tau)], \qquad t \in [0, \infty), \qquad j \in \Z,\label{eqn:fhat_sharp}
\end{gather}
give the single-variable martingale-functions such that, if $\widehat{f}^{\RN}_j(t, B(t); \tau)$, $ t \geq 0$, $j \in \Z$ are f\/inite, then they give martingales with respect to the f\/iltration $\cF_t$ of one-dimensional Brownian motion~$B(t)$, $t \geq 0$,
\begin{gather*}
\rE[ \widehat{f}^{\RN}_j(t, B(t); \tau) | \cF_s] =\widehat{f}^{\RN}_j(s, B(s); \tau), \qquad 0 \leq s < t < \infty \qquad \mbox{a.s.} \quad j \in \Z.
\end{gather*}
We have found that $\widehat{f}^{\RN}_j(t, x; \tau)$'s are expressed using $f^{\RN}_j(t, x; \cdot)$ by shifting
\begin{gather*}
\tau \to \tau - \frac{i \cN^{\RN} t}{2 \pi r^2}
\end{gather*}
with (\ref{eqn:cN_R}) and multiplying time-dependent factors. With $0 < t_{\ast} < \infty$, we put~(\ref{eqn:tau_R}).

\begin{Lemma}\label{thm:f_hat} For $f^{\RN}_j$, $j \in \Z$ given by \eqref{eqn:f}, we obtain the following infinite series of linearly independent martingale-functions
\begin{gather}
\widehat{f}^{\RN}_j\big(t, x; \tau^{\RN}(0)\big)=e^{J^{\RN}(j)^2 t/2 r^2} f^{\RN}_j\big(x; \tau^{\RN}(t)\big),\qquad t \in [0, t_{\ast}), \qquad j \in \Z, \label{eqn:f_hat_sharp}
\end{gather}
where $J^{\RN}(j)$, $j \in \Z$ are given by \eqref{eqn:J_R}.
\end{Lemma}

\begin{proof}By (\ref{eqn:f}), it is enough to calculate
\begin{gather}
\widetilde{\rE}\left[ e^{\pm i J^{\RN}(j)(x+i \widetilde{B}(t))/r}\vartheta_1 \left(J^{\RN}(j) \tau \pm \frac{\cN^{\RN}(x+i \widetilde{B}(t))}{2 \pi r} + \alpha ;\cN^{\RN} t \right) \right] \label{eqn:Etildef1}
\end{gather}
with a constant $\alpha$. By the def\/inition (\ref{eqn:theta}) of $\vartheta_1$, this is equal to
\begin{gather*}
e^{\pm i J^{\RN}(j) x} i \sum_{n \in \Z} (-1)^n e^{\cN^{\RN} \tau \pi i (n-1/2)^2+(J^{\RN}(j) \tau \pm \cN^{\RN}x/2 \pi r + \alpha) \pi i (2n-1)}\\
\qquad{}\times \widetilde{\rE} \big[e^{\mp \{J^{\RN}(j)+ \cN^{\RN}(2n-1)/2\} \widetilde{B}(t)/r}\big].
\end{gather*}
Since
\begin{gather*}
\widetilde{\rE} \big[e^{\mp \{J^{\RN}(j)+ \cN^{\RN}(2n-1)/2\} \widetilde{B}(t)/r} \big]= e^{\{J^{\RN}(j)+\cN^{\RN}(2n-1)/2\}^2 t/2r^2}\\
\qquad {} =e^{J^{\RN}(j)^2 t/2r^2}e^{\cN^{\RN}(-i \cN^{\RN}t/2 \pi r^2) \pi i (n-1/2)^2+J^{\RN}(j)(-i \cN^{\RN}t/2 \pi r^2) \pi i (2n-1)},
\end{gather*}
we can see that (\ref{eqn:Etildef1}) is equal to
\begin{gather*}
e^{J^{\RN}(j)^2 t/2r^2}e^{\pm i J^{\RN}(j) x/r}\vartheta_1 \left(J^{\RN}(j) \left(\tau-\frac{i \cN^{\RN} t}{2 \pi r^2} \right)\pm \frac{\cN^{\RN} x}{2 \pi r} + \alpha ;\cN^{\RN} \left( \tau-\frac{i \cN^{\RN} t}{2 \pi r^2} \right)\right).
\end{gather*}
Therefore, if we set $\tau=i \cN^{\RN} t_{\ast}/2 \pi r^2$, (\ref{eqn:f_hat_sharp}) are obtained.
\end{proof}

We call the multivariate function $D_{\u}(t, \x)$ of $\x \in \R^N$ def\/ined by (\ref{eqn:D1}) the {\it determinantal martingale-function}, since if it is f\/inite, it gives the determinantal martingale when we put the $N$-dimensional Brownian motion $\B(t)$, $t \geq 0$ into $\x$. The following is proved.

\begin{Proposition} \label{thm:det_mar_fun} For $f^{\RN}_j$, $j \in \Z$ given by \eqref{eqn:f}, the determinantal martingale-functions are factorized as follows
\begin{gather}
D^{\RN}_{\u}(t, \x) = \frac{c^{\RN}_0\big(\tau^{\RN}(t)\big)}{c^{\RN}_0\big(\tau^{\RN}(0)\big)}
\frac{k_{\rm sym}^{\RN}\big(\x; \tau^{\RN}(t)\big)}{k^{\RN}_{\rm sym}\big(\u; \tau^{\RN}(0)\big)}
\prod_{\ell=1}^N \frac{k^{\RN}_1\big(x_{\ell}; \tau^{\RN}(t)\big)}{k^{\RN}_1\big(u_{\ell}; \tau^{\RN}(0)\big)}
\prod_{1 \leq j < k \leq N}\frac{k^{\RN}_2\big(x_k, x_j; \tau^{\RN}(t)\big)}{k^{\RN}_2\big(u_k, u_j; \tau^{\RN}(0)\big)},\!\!\!\label{eqn:D_R2}
\end{gather}
where
\begin{gather}
c^{\AN}_0(\tau)=\eta(\tau)^{-(N-1)(N-2)/2}\label{eqn:c_0_AN}
\end{gather}
and for $\RN=\BN,\BNv, \CN, \CNv, \BCN, \DN$, $c^{\RN}_0(\tau)$ are given by \eqref{eqn:c_0}. For $\RN=\BN, \BNv, \CN, \CNv$, $\BCN, \DN$, \eqref{eqn:D_R2} are written as \eqref{eqn:D_Ra}--\eqref{eqn:D_Rc} with \eqref{eqn:c_0} and \eqref{eqn:c_1_2_R}.
\end{Proposition}

\begin{proof} By the second equality of (\ref{eqn:D3}) and Lemma \ref{thm:f_hat},
\begin{gather*}
D^{\RN}_{\u}(t, \x)= \frac{{\det\limits_{1 \leq j, k \leq N}\big[e^{J^{\RN}(j)^2 t/2r^2} f^{\RN}_j(x_k; \tau^{\RN}(t))\big]}}
{{\det\limits_{1 \leq j, k \leq N} \big[f^{\RN}_j(u_k; \tau^{\RN}(0))\big]}}
= e^{(t/2r^2) \sum\limits_{j=1}^N J^{\RN}(j)^2}
\frac{{\det\limits_{1 \leq j, k \leq N}\big[f^{\RN}_j(x_k; \tau^{\RN}(t))\big]}}
{{\det\limits_{1 \leq j, k \leq N} \big[f^{\RN}_j(u_k; \tau^{\RN}(0))\big]}},
\end{gather*}
where the multi-linearity of determinant was used in the second equality. Then applying Lemma~\ref{thm:RS}, we obtain~(\ref{eqn:D_R2}) with
\begin{gather*}
\frac{c^{\RN}_0(\tau^{\RN}(t))}{c^{\RN}_0(\tau^{\RN}(0))} = e^{(t/2r^2) \sum\limits_{j=1}^N J^{\RN}(j)^2} \frac{k^{\RN}_0(\tau^{\RN}(t))}{k^{\RN}_0(\tau^{\RN}(0))}.
\end{gather*}
By (\ref{eqn:k0}) and the summation formulas
\begin{gather*}
\sum\limits_{j=1}^N j=N(N+1)/2, \qquad \sum\limits_{j=1}^N j^2 = N(N+1)(2N+1)/6,
 \end{gather*} we obtain the expressions~(\ref{eqn:c_0}) using the Dedekind modular function~(\ref{eqn:Dedekind1}), since $\eta(\tau)=e^{\tau \pi i/12} q_0(\tau)$ with~(\ref{eqn:q_q0}). Thus we can derive readily the expressions (\ref{eqn:D_Ra})--(\ref{eqn:D_Rc}) with~(\ref{eqn:c_0}) and~(\ref{eqn:c_1_2_R}) from~(\ref{eqn:D_R2}).
\end{proof}

\section{Proof of Theorem \ref{thm:main1}} \label{sec:proof1}
Assume $t \in [0, t_{\ast})$. By (\ref{eqn:theta_positive}), we see that if $0 \leq x \leq 2 \pi r$, $0 \leq k_1^{\RN}\big(x; \tau^{\RN}(t)\big) < \infty$ for $\RN=\BN, \CNv, \BCN$; if $0 \leq x \leq \pi r$, $0 \leq k_1^{\RN}\big(x; \tau^{\RN}(t)\big) < \infty$ for $\RN=\BNv, \CN$; and $k_1^{\DN} \equiv 1$. We also see that if
\begin{gather}
|x_1| < x_2 < \cdots < x_{N-1} < x_N \wedge (2 \pi r -x_N),\label{eqn:positive1a}
\end{gather}
then \looseness=-1 $0 < \prod\limits_{1 \leq j < k \leq N} k_2^{\RN}\big(x_k, x_j; \tau^{\RN}(t)\big) < \infty$. The inequality $|x_1|<x_2$ means that $x_2$ is greater than both of $x_1$ and its ref\/lection at $0$, and $x_{N-1} < x_N \wedge (2 \pi r -x_N)$ means that $x_{N-1}$ is smaller than both of $x_N$ and its ref\/lection at~$\pi r$. The above observation implies that if $\u \in \W_N^{(0, \pi r)}$, then
\begin{gather*}
\begin{split}
& 0 \leq \prod_{\ell=1}^N k^{\RN}_1\big(B_{\ell}(t); \tau^{\RN}(t)\big)\prod_{1 \leq j < k \leq N} k^{\RN}_2\big(B_k(t), B_j(t); \tau^{\RN}(t)\big) < \infty,\\
& \text{for} \quad 0 \leq t \leq T_{\rm collision}, \qquad \text{a.s.\ in $\rP^{[0, \pi r]}_{\u}$}.
\end{split}
\end{gather*}

Due to the Karlin--McGregor--Lindstr\"om--Gessel--Viennot (KMLGV) formula \cite{GV85, KM59,Lin73} for non-intersecting paths, the transition probability density of the $N$ particle system in $[0, \pi r]$ governed by the probability law $\P^{\RN}_{\u}$ def\/ined by~(\ref{eqn:Plaw_R1}) is given by
\begin{gather}
p^{\RN}(t, \y|s, \x) = \frac{c^{\RN}_0\big(\tau^{\RN}(t)\big)}{c^{\RN}_0\big(\tau^{\RN}(s)\big)}\prod_{\ell=1}^N \frac{k^{\RN}_1\big(y_{\ell}; \tau^{\RN}(t)\big)} {k^{\RN}_1\big(x_{\ell}; \tau^{\RN}(s)\big)}\prod_{1 \leq j < k \leq N} \frac{k^{\RN}_2\big(y_k, y_j; \tau^{\RN}(t)\big)}{k^{\RN}_2\big(x_k, x_j; \tau^{\RN}(s)\big)}\nonumber\\
\hphantom{p^{\RN}(t, \y|s, \x) =}{} \times
\det_{1 \leq n, n' \leq N} \big[ p^{[0, \pi r]}(t-s, y_n|x_{n'}) \big], \qquad 0 \leq s < t < t_{\ast}, \quad \x, \y \in \W_N^{(0, \pi r)}.\!\!\!\label{eqn:p_R1}
\end{gather}
Since we consider the measure-valued stochastic process~(\ref{eqn:Xi_R1}), the conf\/iguration is unlabeled and hence all the observables at each time should be symmetric functions of particle positions. For an arbitrary number $M \in \N$ and an arbitrary strictly increasing series of times $0 \equiv t_0 < t_1 < \cdots < t_M < t_{\ast}$, let $g_{t_m}(\x)$, $m =1,2, \dots, M$, be symmetric functions of $\x \in \W_N^{(0, \pi r)}$. Then~(\ref{eqn:Plaw_R1}) and~(\ref{eqn:p_R1}) give
\begin{gather}
 \E^{\RN}_{\u} \left[ \prod_{m=1}^M g_{t_m}(\X(t_m)) \right]=\rE_{\u}^{[0, \pi r]} \left[
\prod_{m=1}^M g_{t_m}(\B(t_m)){\bf 1}(T_{\rm collision} > t_M ) D^{\RN}_{\u}(t, \B(t_M)) \right]\nonumber\\
{} =\int_{\W_N^{(0, \pi r)}} d \x^{(1)} \cdots \int_{\W_N^{(0, \pi r)}} d \x^{(M)}\prod_{m=1}^M
\det_{1 \leq n, {n'} \leq N} \big[ p^{[0, \pi r]} \big(t_m-t_{m-1}, x^{(m)}_n | x^{(m-1)}_{n'}\big) \big]g_{t_m}\big(\x^{(m)}\big)\nonumber\\
\quad{} \times \frac{c^{\RN}_0\big(\tau^{\RN}(t_M)\big)}{c^{\RN}_0\big(\tau^{\RN}(0)\big)}
\prod_{\ell=1}^N \frac{k^{\RN}_1\big(x^{(M)}_{\ell}; \tau^{\RN}(t_M)\big)}{k^{\RN}_1\big(u_{\ell}; \tau^{\RN}(0)\big)}
\prod_{1 \leq j < k \leq N} \frac{k^{\RN}_2\big(x^{(M)}_k, x^{(M)}_j; \tau^{\RN}(t_M)\big)}
{k^{\RN}_2\big(u_k, u_j; \tau^{\RN}(0)\big)}.\label{eqn:E_eq1}
\end{gather}
By def\/inition of determinant, the above is equal to
\begin{gather*}
 \sum_{\sigma_1 \in \cS_N} \cdots \sum_{\sigma_M \in \cS_N}\int_{\W_N^{(0, \pi r)}} d \x^{(1)} \cdots\int_{\W_N^{(0, \pi r)}} d \x^{(M)} \\
\qquad {}\times
\prod_{m=1}^M \left\{{\rm sgn}(\sigma_m)\prod_{n=1}^Np^{[0, \pi r]} \big(t_m-t_{m-1}, x^{(m)}_{\sigma^{(m)}(n)} | x^{(m-1)}_{\sigma^{(m-1)}(n)}\big)
g_{t_m}\big(\x^{(m)}\big)\right. \\
 \qquad \qquad {} \times \frac{c^{\RN}_0\big(\tau^{\RN}(t_m)\big)}{c^{\RN}_0\big(\tau^{\RN}(t_{m-1})\big)}
\prod_{\ell=1}^N \frac{k^{\RN}_1\big(x^{(m)}_{\ell}; \tau^{\RN}(t_m)\big)} {k^{\RN}_1\big(x^{(m-1)}_{\ell}; \tau^{\RN}(t_{m-1})\big)} \\
 \left.\qquad \qquad{}
\times \left| \prod_{1 \leq j < k \leq N} \frac{k^{\RN}_2\big(x^{(m)}_{\sigma^{(m)}(k)}, x^{(m)}_{\sigma^{(m)}(j)}; \tau^{\RN}(t_m)\big)}
{k^{\RN}_2\big(x^{(m-1)}_{\sigma^{(m-1)}(k)}, x^{(m-1)}_{\sigma^{(m-1)}(j)}; \tau^{\RN}(t_{m-1})\big)} \right| \right\},
\end{gather*}
where $\cS_N$ denotes the collection of all permutations of $N$ indices $\{1,2, \dots, N\}$, $\sigma^{(m)} \equiv \sigma_1 \circ \sigma_2 \circ \cdots \circ \sigma_m$, $m \geq 1$, $\sigma^{(0)}={\rm id}$, and $x^{(0)}_j=u_j$, $j=1,\dots, N$. Here we used the fact that $g_{t_m}(\x)$'s and $\prod\limits_{\ell=1}^N k^{\RN}_1(x_{\ell}; \cdot)$ are symmetric functions of~$\x$. By the def\/inition of $\W_N^{(0, \pi r)}$ and the fact that
\begin{gather*}
\prod_{1 \leq j < k \leq N} k^{\RN}_2(x_{\sigma(k)}, x_{\sigma(j)}; \cdot) = {\rm sgn}(\sigma)\prod_{1 \leq j < k \leq N} k^{\RN}_2(x_k, x_j; \cdot),\qquad \sigma \in \cS_N,
\end{gather*}
this is equal to
\begin{gather*}
\int_{[0, \pi r]^N} d \x^{(1)} \cdots \int_{[0, \pi r]^N} d \x^{(M)}\prod_{m=1}^M \left\{\prod_{n=1}^N
p^{[0, \pi r]} \big(t_m-t_{m-1}, x^{(m)}_n | x^{(m-1)}_n\big)g_{t_m}\big(\x^{(m)}\big) \right\}\\
\qquad {} \times \frac{c^{\RN}_0\big(\tau^{\RN}(t_M)\big)}{c^{\RN}_0\big(\tau^{\RN}(0)\big)}
\prod_{\ell=1}^N \frac{k^{\RN}_1\big(x^{(M)}_{\ell}; \tau^{\RN}(t_M)\big)}
{k^{\RN}_1\big(u_{\ell}; \tau^{\RN}(0)\big)}\prod_{1 \leq j < k \leq N} \frac{k^{\RN}_2\big(x^{(M)}_k, x^{(M)}_j; \tau^{\RN}(t_M)\big)}
{k^{\RN}_2\big(u_k, u_j; \tau^{\RN}(0)\big)}.
\end{gather*}
Therefore, we obtain the equality
\begin{gather}
\E^{\RN}_{\u} \left[ \prod_{m=1}^M g_{t_m}(\X(t_m)) \right]= \rE^{[0, \pi r]}_{\u} \left[\prod_{m=1}^M g_{t_m}(\B(t_m)) D^{\RN}_{\u}(t_M, \B(t_M)) \right],\label{eqn:E_eq2}
\end{gather}
that is, the noncolliding condition, $T_{\rm collision} > t_M$, can be omitted in (\ref{eqn:E_eq1}). The expression (\ref{eqn:E_eq2}) is interpreted as the expectation of the product of symmetric functions $\prod\limits_{m=1}^M g_{t_m}(\cdot)$ with respect to the signed measure
\begin{gather}
\widehat{\P}^{\RN}_{\u} \big|_{\cF_t} = D^{\RN}_{\u}(t, \B(t)) \rP^{[0, \pi r]}_{\u} \big|_{\cF_t},\qquad t \in [0, t_{\ast}), \label{eqn:signed_meas1}
\end{gather}
which is a modif\/ication of (\ref{eqn:Phat1}) obtained by replacing $\rP_{\u}$ by $\rP_{\u}^{[0, \pi r]}$ and $D_{\u}$ by $D^{\RN}_{\u}$. Then we can apply Proposition \ref{thm:Fredholm}, where we replace $p_{\rm BM}$ by $p^{[0, \pi r]}$ and use the proper martingale functions given by~(\ref{eqn:MB1}) with (\ref{eqn:Phi_Ra})--(\ref{eqn:Phi_Rc}). Thus the proof is completed.

\section{Proof of Theorem \ref{thm:main2}} \label{sec:proof2}

Theorem \ref{thm:main2} is concluded from the following key lemma.

\begin{Lemma}\label{thm:Kolmogorov}
The transition probability densities $p^{\BN}$, $p^{\CN}$ and $p^{\DN}$ given by~\eqref{eqn:p_R1} for $\RN=\BN,\CN$, and~$\DN$ solve the following backward Kolmogorov equations
\begin{gather}
 - \frac{\partial p^{\BN}(t, \y|s, \x)}{\partial s}= \frac{1}{2} \sum_{j=1}^N \frac{\partial^2 p^{\BN}(t, \y|s, \x)}{\partial x_j^2}
+ \sum_{j=1}^N A^{2 \pi r}_{2N-1}(t_{\ast}-s, x_j)\frac{\partial p^{\BN}(t, \y|s, \x)}{\partial x_j}\nonumber\\
\qquad {} + \sum_{\substack{1 \leq j, k \leq N, \cr j \not= k}}
\big(A^{2 \pi r}_{2N-1}(t_{\ast}-s, x_j-x_k) + A^{2 \pi r}_{2N-1}(t_{\ast}-s, x_j+x_k) \big)
\frac{\partial p^{\BN}(t, \y|s, \x)}{\partial x_j},\label{eqn:Kolm_B1}\\
- \frac{\partial p^{\CN}(t, \y|s, \x)}{\partial s}
= \frac{1}{2} \sum_{j=1}^N \frac{\partial^2 p^{\CN}(t, \y|s, \x)}{\partial x_j^2}
+2 \sum_{j=1}^N A^{2 \pi r}_{2(N+1)}(t_{\ast}-s, 2 x_j)\frac{\partial p^{\CN}(t, \y|s, \x)}{\partial x_j}\nonumber\\
\qquad{} +\sum_{\substack{1 \leq j, k \leq N, \cr j \not= k}}
\big(A^{2 \pi r}_{2(N+1)}(t_{\ast}-s, x_j-x_k) + A^{2 \pi r}_{2(N+1)}(t_{\ast}-s, x_j+x_k)\big )
\frac{\partial p^{\CN}(t, \y|s, \x)}{\partial x_j},\!\!\!\label{eqn:Kolm_C1}\\
 - \frac{\partial p^{\DN}(t, \y|s, \x)}{\partial s}
= \frac{1}{2} \sum_{j=1}^N \frac{\partial^2 p^{\DN}(t, \y|s, \x)}{\partial x_j^2}\nonumber\\
\qquad {}+\sum_{\substack{1 \leq j, k \leq N, \cr j \not= k}}
\big( A^{2 \pi r}_{2(N-1)}(t_{\ast}-s, x_j-x_k) + A^{2 \pi r}_{2(N-1)}(t_{\ast}-s, x_j+x_k) \big)
\frac{\partial p^{\DN}(t, \y|s, \x)}{\partial x_j},\!\!\!\!\label{eqn:Kolm_D1}
\end{gather}
for $\x, \y \in \W_N^{(0, \pi r)}$, $0 \leq s < t < t_{\ast}$, under the conditions
\begin{gather}
\lim_{s \uparrow t} p^{\BN}(t, \y|s, \x)= \lim_{s \uparrow t} p^{\CN}(t, \y|s, \x)=\lim_{s \uparrow t} p^{\DN}(t, \y|s, \x)=\prod_{j=1}^N \delta(x_j-y_j).
\label{eqn:Kolm_initial}
\end{gather}
\end{Lemma}

\begin{proof}For given $t \in (0, t_{\ast})$ and $\y \in \W_N^{(0, \pi r)}$, put
\begin{gather}
u(s, \x)=w(s, \x) q^{[0, \pi r]}(t-s, \y|\x), \qquad 0 < s < t, \qquad \x \in \W_N^{(0, \pi r)} \label{eqn:usx1}
\end{gather}
with the KMLGV determinant
\begin{gather}
q^{[0, \pi r]}(t-s, \y|\x)=\det_{1 \leq j, k \leq N} \big[ p^{[0, \pi r]}(t-s, y_j|x_k)\big],\label{eqn:qN1}
\end{gather}
where $w(s, \x)$ is a $\rC^{1,2}$-function which will be specif\/ied later, and $p^{[0, \pi r]}=p^{[0, \pi r]}_{\rm aa}$ given by (\ref{eqn:p_abs}) for $\RN=\CN$, $p^{[0, \pi r]}=p^{[0, \pi r]}_{\rm rr}$ given by (\ref{eqn:p_ref}) for $\RN=\DN$. Since $p^{[0, \pi r]}$ is the transition probability density
of a boundary-conditioned Brownian motion, it solves the dif\/fusion equation
\begin{gather*}
-\frac{\partial p^{[0, \pi r]}(t-s, y|x)}{\partial s} =\frac{1}{2} \frac{\partial^2 p^{[0, \pi r]}(t-s, y|x)}{\partial x^2}
\end{gather*}
with $\lim\limits_{s \uparrow t} p^{[0, \pi r]}(t-s, y|x)=\delta(x-y)$. The KMLGV determinant (\ref{eqn:qN1}) is the summation of products of $p^{[0, \pi r]}$'s, and hence
\begin{gather*}
- \frac{\partial q(t-s, \y|\x)}{\partial s}= \frac{1}{2} \sum_{j=1}^N \frac{\partial^2 q(t-s, \y|\x)}{\partial x_j^2}
\end{gather*}
is satisf\/ied. It is easy to verify that
\begin{gather}
\lim_{s \uparrow t} q(t-s, \y|\x)=
\prod_{j=1}^N \delta(x_j-y_j).
\label{eqn:qN3}
\end{gather}
Therefore, the following equation holds
\begin{gather*}
\left\{ \frac{\partial}{\partial s}+ \frac{1}{2} \sum_{j=1}^N \frac{\partial^2}{\partial x_j^2} \right\} u(s, \x)\\
 \qquad{} = q^{[0, \pi r]}(t-s, \y|\x) \left\{ \frac{\partial}{\partial s}+ \frac{1}{2} \sum_{j=1}^N \frac{\partial^2}{\partial x_j^2} \right\} w(s, \x)
+\sum_{j=1}^N \frac{\partial w(s, \x)}{\partial x_j}\frac{\partial q^{[0, \pi r]}(t-s, \y|\x)}{\partial x_j}.
\end{gather*}
Since
\begin{gather*}
\sum_{j=1}^N \frac{\partial w(s, \x)}{\partial x_j}\frac{\partial u(s, \x)}{\partial x_j}\\
 \qquad{} = q^{[0, \pi r]}(t-s, \y|\x) \sum_{j=1}^N \left( \frac{\partial w(s, \x)}{\partial x_j} \right)^2
+ w(s, \x) \sum_{j=1}^N \frac{\partial w(s, \x)}{\partial x_j}\frac{\partial q^{[0, \pi r]}(t-s, \y|\x)}{\partial x_j},
\end{gather*}
the above equation is written as
\begin{gather}
\left\{ \frac{\partial}{\partial s}+\frac{1}{2} \sum_{j=1}^N \frac{\partial^2}{\partial x_j^2}- \frac{1}{w(s, \x)}
\sum_{j=1}^N \frac{\partial w(s, \x)}{\partial x_j} \frac{\partial}{\partial x_j}\right\} u(s, \x)\nonumber\\
\qquad{} = q^{[0, \pi r]}(t-s, \y|\x)\left[\left\{ \frac{\partial}{\partial s}+ \frac{1}{2} \sum_{j=1}^N \frac{\partial^2}{\partial x_j^2}\right\} w(s, \x)
- \frac{1}{w(s, \x)} \sum_{j=1}^N \left( \frac{\partial w(s, \x)}{\partial x_j} \right)^2\right].\label{eqn:main_eq}
\end{gather}
Now we put
\begin{gather}
w(s, \x) = w^{\RN}(s, \x)\label{eqn:fsx1}\\
\equiv \left\{g^{\RN}(s)\prod_{\ell=1}^N \vartheta_1 \left( \frac{ c_1^{\RN} x_{\ell}}{2 \pi r}; \tau^{\RN}(s) \right)
\prod_{1 \leq j < k \leq N}\vartheta_1 \left( \frac{x_k-x_j}{2\pi r} ; \tau^{\RN}(s) \right)
\vartheta_1 \left( \frac{x_k+x_j}{2 \pi r}; \tau^{\RN}(s) \right) \right\}^{-1}\nonumber
\end{gather}
for $\RN=\BN, \CN$, and
\begin{gather}
w(s, \x) = w^{\DN}(s, \x)\nonumber\\
\hphantom{w(s, \x)}{} \equiv \left\{ g^{\DN}(s)\prod_{1 \leq j < k \leq N}
\vartheta_1 \left( \frac{x_k-x_j}{2\pi r} ; \tau^{\DN}(s) \right)\vartheta_1 \left( \frac{x_k+x_j}{2 \pi r}; \tau^{\DN}(s) \right) \right\}^{-1}\label{eqn:fsx1b}
\end{gather}
for $\RN=\DN$, where $g^{\RN}(s)$, $\RN=\BN, \CN, \DN$, are $\rC^1$-functions of $s$, which will be determined later, and $c_1^{\BN}=1$, $c_1^{\CN}=2$. We def\/ine $\alpha^{\RN}=1$ for $\RN=\BN$ and $\CN$, and $\alpha^{\RN}=0$ for $\RN=\DN$. We put $c_1^{\DN}=1$ for convention. By (\ref{eqn:A1}) with (\ref{eqn:A_odd}), we see that
\begin{gather}
- \frac{1}{w^{\RN}(s, \x)} \frac{\partial w^{\RN}(s, \x)}{\partial x_j}=\alpha^{\RN} c_1^{\RN} A_{\cN^{\RN}}^{2 \pi r}\big(t_{\ast}-s, c_1^{\RN} x_j\big)
\nonumber\\
\hphantom{- \frac{1}{w^{\RN}(s, \x)} \frac{\partial w^{\RN}(s, \x)}{\partial x_j}=}{}
+\sum_{\substack{1 \leq k \leq N, \cr k \not=j}} \big( A_{\cN^{\RN}}^{2 \pi r}(t_{\ast}-s, x_j-x_k) +A_{\cN^{\RN}}^{2 \pi r}(t_{\ast}-s, x_j+ x_k) \big), \label{eqn:dfdx1}
\end{gather}
$j=1,2, \dots, N$, and if we multiply (\ref{eqn:dfdx1}) by $\partial p^{\RN}(t, \y|s, \x)/\partial x_j$ and take summation over $j=1,2, \dots, N$ for $\RN=\BN, \CN, \DN$, they give the drift terms in (\ref{eqn:Kolm_B1})--(\ref{eqn:Kolm_D1}). Therefore, if we can choose $g^{\RN}(s)$ so that the equation
\begin{gather}
\left\{ \frac{\partial}{\partial s}+ \frac{1}{2} \sum_{j=1}^N \frac{\partial^2}{\partial x_j^2} \right\} w^{\RN}(s, \x)
- \frac{1}{w^{\RN}(s, \x)} \sum_{j=1}^N\left( \frac{\partial w^{\RN}(s, \x)}{\partial x_j} \right)^2=0 \label{eqn:eqA1}
\end{gather}
holds, then the r.h.s.\ of (\ref{eqn:main_eq}) vanishes and we can conclude that $u(s, \x)$ given in the form (\ref{eqn:usx1}) with (\ref{eqn:fsx1}) and (\ref{eqn:fsx1b}) solve the Kolmogorov equations (\ref{eqn:Kolm_B1})--(\ref{eqn:Kolm_D1}).

From (\ref{eqn:fsx1}) and (\ref{eqn:fsx1b}) with (\ref{eqn:tau_R}), we f\/ind that
\begin{gather*}
\frac{1}{w^{\RN}(s, \x)} \frac{\partial w^{\RN}(s, \x)}{\partial s}= -\frac{1}{g^{\RN}(s)} \frac{d g^{\RN}(s)}{ds}
+\alpha^{\RN} \frac{i \cN^{\RN}}{2 \pi r^2} \sum_{j=1}^N \frac{\dot{\vartheta_1}\big(c_1^{\RN} x_j/2 \pi r; \tau^{\RN}(s)\big)}
{\vartheta_1\big(c_1^{\RN} x_j/2 \pi r; \tau^{\RN}(s)\big)}\\
 \qquad{} + \frac{1}{2} \frac{i \cN^{\RN}}{2 \pi r^2} \sum_{\substack{1 \leq j, k \leq N, \cr j \not= k}}
\left\{\frac{\dot{\vartheta_1}\big((x_j-x_k)/2 \pi r; \tau^{\RN}(s)\big)}{\vartheta_1\big((x_j-x_k)/2 \pi r; \tau^{\RN}(s)\big)}
+ \frac{\dot{\vartheta_1}\big((x_j+x_k)/2 \pi r; \tau^{\RN}(s)\big)}{\vartheta_1\big((x_j+x_k)/2 \pi r; \tau^{\RN}(s)\big)} \right\},
\end{gather*}
where $\dot{\vartheta}_1(v; \tau)=\partial \vartheta_1(v; \tau)/\partial \tau$. If we use the equation (\ref{eqn:Jacobi_eq}), then the above is written as
\begin{gather*}
 \frac{1}{w^{\RN}(s, \x)} \frac{\partial w^{\RN}(s, \x)}{\partial s} = -\frac{1}{g^{\RN}(s)} \frac{d g^{\RN}(s)}{ds}
+\alpha^{\RN} \frac{\cN^{\RN}}{8 \pi^2 r^2} \sum_{j=1}^N\frac{\vartheta_1''\big(c_1^{\RN} x_j/2 \pi r; \tau^{\RN}(s)\big)}{\vartheta_1\big(c_1^{\RN} x_j/2 \pi r; \tau^{\RN}(s)\big)}\\
\qquad {}+ \frac{1}{2} \frac{\cN^{\RN}}{8 \pi^2 r^2}\sum_{\substack{1 \leq j, k \leq N, \cr j \not= k}}
\left\{\frac{\vartheta_1''\big((x_j-x_k)/2 \pi r; \tau^{\RN}(s)\big)}{\vartheta_1\big((x_j-x_k)/2 \pi r; \tau^{\RN}(s)\big)}
+ \frac{\vartheta_1''\big((x_j+x_k)/2 \pi r; \tau^{\RN}(s)\big)}{\vartheta_1\big((x_j+x_k)/2 \pi r; \tau^{\RN}(s)\big)} \right\},
\end{gather*}
where $\vartheta_1''(v; \tau)=\partial^2 \vartheta_1(v; \tau)/\partial v^2$. From (\ref{eqn:dfdx1}), we f\/ind that
\begin{gather*}
\frac{\partial^2 w^{\RN}(s, \x)}{\partial x_j^2}= - w^{\RN}(s, \x)\frac{\partial}{\partial x_j}
\Bigg\{ \alpha^{\RN} c_1^{\RN} A_{\cN^{\RN}}^{2 \pi r}\big(t_{\ast}-s, c_1^{\RN} x_j\big)\\
{} + \sum_{\substack{1 \leq k \leq N, \cr k \not=j}}
\big(A_{\cN^{\RN}}^{2 \pi r}(t_{\ast}-s, x_j-x_k)+A_{\cN^{\RN}}^{2 \pi r}(t_{\ast}-s, x_j+x_k) \big)\Bigg\}
+ \frac{1}{w^{\RN}(s, \x)} \left( \frac{\partial w^{\RN}(s, \x)}{\partial x_j} \right)^2,
\end{gather*}
and hence
\begin{gather*}
\frac{1}{2} \sum_{j=1}^N \frac{\partial^2 w^{\RN}(s, \x)}{\partial x_j^2}- \frac{1}{w^{\RN}(s, \x)} \sum_{j=1}^N
\left( \frac{\partial w^{\RN}(s, \x)}{\partial x_j} \right)^2= - \frac{1}{2} w^{\RN}(s, \x) S
\end{gather*}
with
\begin{gather*}
S =\sum_{j=1}^N \Bigg[\frac{\partial}{\partial x_j}\bigg\{ \alpha^{\RN} c_1^{\RN} A_{\cN^{\RN}}^{2 \pi r}\big(t_{\ast}-s, c_1^{\RN} x_j\big)\\
{} + \sum_{\substack{1 \leq k \leq N, \cr k \not=j}}
\big(A_{\cN^{\RN}}^{2 \pi r}(t_{\ast}-s, x_j-x_k)+A_{\cN^{\RN}}^{2 \pi r}(t_{\ast}-s, x_j+x_k) \big)\bigg\}\\
{}+ \bigg\{ \alpha^{\RN} c_1^{\RN} A_{\cN^{\RN}}^{2 \pi r}\big(t_{\ast}-s, c_1^{\RN} x_j\big)
+ \sum_{\substack{1 \leq k \leq N, \cr k \not=j}}
\big(A_{\cN^{\RN}}^{2 \pi r}(t_{\ast}-s, x_j-x_k)+A_{\cN^{\RN}}^{2 \pi r}(t_{\ast}-s, x_j+x_k) \big)\bigg\}^2\Bigg].
\end{gather*}
By def\/inition (\ref{eqn:A1}),
\begin{gather}
\frac{\partial A_{\cN^{\RN}}^{2 \pi r}\big(t_{\ast}-s, c_1^{\RN} x\big)}{\partial x}= \frac{c_1^{\RN}}{(2 \pi r)^2}
\frac{\vartheta_1''\big(c_1^{\RN} x/2 \pi r; \tau^{\RN}(s)\big)}{\vartheta_1\big(c_1^{\RN} x/2 \pi r; \tau^{\RN}(s)\big)}
-c_1^{\RN} \big(A_{\cN^{\RN}}^{2 \pi r}\big(t_{\ast}-s, c_1^{\RN} x\big)\big)^2.\label{eqn:dAdx1}
\end{gather}
Hence
\begin{gather*}
S =\alpha^{\RN} \frac{\big(c_1^{\RN}\big)^2}{(2 \pi r)^2}\sum_{j=1}^N\frac{\vartheta_1''\big(c_1^{\RN} x_j/2 \pi r; \tau^{\RN}(s)\big)}
{\vartheta_1\big(c_1^{\RN} x_j/2 \pi r; \tau^{\RN}(s)\big)}\\
\hphantom{S=}{} + \frac{1}{(2 \pi r)^2}\sum_{\substack{1 \leq j, k \leq N, \cr j \not= k}}\left\{
\frac{\vartheta_1''\big((x_j-x_k)/2 \pi r; \tau^{\RN}(s)\big)}{\vartheta_1\big((x_j-x_k)/2 \pi r; \tau^{\RN}(s)\big)}
+ \frac{\vartheta_1''\big((x_j+x_k)/2 \pi r; \tau^{\RN}(s)\big)}{\vartheta_1\big((x_j+x_k)/2 \pi r; \tau^{\RN}(s)\big)} \right\}\\
\hphantom{S=}{}+ 2 \alpha^{\RN} c_1^{\RN} \sum_{\substack{1 \leq j, k \leq N, \cr j \not= k}}
A_{\cN^{\RN}}^{2 \pi r}\big(t_{\ast}-s, c_1^{\RN} x_j\big)\big(A_{\cN^{\RN}}^{2 \pi r}(t_{\ast}-s, x_j-x_k)+A_{\cN^{\RN}}^{2 \pi r}(t_{\ast}-s, x_j+x_k) \big)\\
\hphantom{S=}{}+ \sum_{\substack{ 1 \leq j, k, \ell \leq N, \cr j \not=k \not=\ell \not= j}}
\big(A_{\cN^{\RN}}^{2 \pi r}(t_{\ast}-s, x_j-x_k)+A_{\cN^{\RN}}^{2 \pi r}(t_{\ast}-s, x_j+x_k) \big)\\
\hphantom{S=}{} \times \big(A_{\cN^{\RN}}^{2 \pi r}(t_{\ast}-s, x_j-x_{\ell})+A_{\cN^{\RN}}^{2 \pi r}(t_{\ast}-s, x_j+x_{\ell}) \big),
\end{gather*}
where we have used the fact
\begin{gather*}
\sum_{\substack{ 1 \leq j, k \leq N, \cr j \not=k}}A_{\cN^{\RN}}^{2 \pi r}(t_{\ast}-s, x_j-x_k)A_{\cN^{\RN}}^{2 \pi r}(t_{\ast}-s, x_j+x_{k})=0
\end{gather*}
concluded from (\ref{eqn:A_odd}). Then (\ref{eqn:eqA1}) holds, if
\begin{gather}
-\frac{1}{g^{\RN}(s)} \frac{d g^{\RN}(s)}{ds} +\alpha^{\RN} \frac{\cN^{\RN}-\big(c_1^{\RN}\big)^2}{8 \pi^2 r^2}\sum_{j=1}^N
\frac{\vartheta_1''\big(c_1^{\RN} x_j/2 \pi r; \tau^{\RN}(s)\big)}{\vartheta_1\big(c_1^{\RN} x_j/2 \pi r; \tau^{\RN}(s)\big)}\nonumber\\
 \qquad{} + \frac{\cN^{\RN}-2}{16 \pi^2 r^2} \sum_{\substack{1 \leq j, k \leq N, \cr j \not= k}}\left\{
\frac{\vartheta_1''\big((x_j-x_k)/2 \pi r; \tau^{\RN}(s)\big)}{\vartheta_1\big((x_j-x_k)/2 \pi r; \tau^{\RN}(s)\big)}+ \frac{\vartheta_1''\big((x_j+x_k)/2 \pi r; \tau^{\RN}(s)\big)}{\vartheta_1\big((x_j+x_k)/2 \pi r; \tau^{\RN}(s)\big)} \right\}\nonumber\\
 \qquad{} - \alpha^{\RN} c_1^{\RN} \sum_{\substack{1 \leq j, k \leq N, \cr j \not= k}}
A_{\cN^{\RN}}^{2 \pi r}\big(t_{\ast}-s, c_1^{\RN} x_j\big) \big(A_{\cN^{\RN}}^{2 \pi r}(t_{\ast}-s, x_j-x_k)+A_{\cN^{\RN}}^{2 \pi r}(t_{\ast}-s, x_j+x_k) \big)\nonumber\\
 \qquad {}-\frac{1}{2} \sum_{\substack{ 1 \leq j, k, \ell \leq N, \cr j \not=k \not=\ell \not= j}}
\big(A_{\cN^{\RN}}^{2 \pi r}(t_{\ast}-s, x_j-x_k)+A_{\cN^{\RN}}^{2 \pi r}(t_{\ast}-s, x_j+x_k) \big)\nonumber\\
\qquad{} \times \big(A_{\cN^{\RN}}^{2 \pi r}(t_{\ast}-s, x_j-x_{\ell})+A_{\cN^{\RN}}^{2 \pi r}(t_{\ast}-s, x_j+x_{\ell}) \big)=0.\label{eqn:eqA2}
\end{gather}

Now we use the following expressions for $A_{\cN^{\RN}}^{2 \pi r}$ and its spatial derivative \cite{Kat15,Kat16}
\begin{gather}
A_{\cN^{\RN}}^{2 \pi r}\big(t_{\ast}-s, c_1^{\RN} x\big)= \zeta_{\cN^{\RN}}\big(t_{\ast}-s, c_1^{\RN} x\big) - c_1^{\RN} \frac{\eta_{\cN^{\RN}}^1(t_{\ast}-s) x}{\pi r},\nonumber\\
\frac{\partial A_{\cN^{\RN}}^{2 \pi r}\big(t_{\ast}-s, c_1^{\RN} x\big)}{\partial x}= -c_1^{\RN} \wp_{\cN^{\RN}}\big(t_{\ast}-s, c_1^{\RN} x\big)-c_1^{\RN} \frac{\eta_{\cN^{\RN}}^1(t_{\ast}-s)}{\pi r}.\label{eqn:A_B1}
\end{gather}
Here $\eta_{\cN^{\RN}}^1(t_{\ast}-s)$ is given by (\ref{eqn:eta1_1}) and we put
\begin{gather}
\zeta_{\cN^{\RN}}(t_{\ast}-s, x)= \zeta(x | 2 \omega_1, 2 \omega_3) \big|_{\omega_1=\pi r, \omega_3= i \cN^{\RN}(t_{\ast}-s)/2r},\nonumber\\
\wp_{\cN^{\RN}}(t_{\ast}-s, x)= \wp(x | 2 \omega_1, 2 \omega_3) \big|_{\omega_1=\pi r, \omega_3=i \cN^{\RN}(t_{\ast}-s)/2r},\label{eqn:Weierstrass1}
\end{gather}
where the Weierstrass $\wp$ function and zeta function $\zeta$ are def\/ined by~(\ref{eqn:wp_zeta}) in Appendix~\ref{sec:appendixA_5}. Applying~(\ref{eqn:A_B1}) to~(\ref{eqn:dAdx1}) gives
\begin{gather}
\frac{1}{(2 \pi r)^2}\frac{\vartheta_1''\big(c_1^{\RN} x/2 \pi r; \tau^{\RN}(s)\big)}{\vartheta_1\big(c_1^{\RN} x/2\pi r; \tau^{\RN}(s)\big)}
= \left( \zeta_{\cN^{\RN}}\big(t_{\ast}-s, c_1^{\RN} x\big) - c_1^{\RN} \frac{\eta_{\cN^{\RN}}^1(t_{\ast}-s) x}{\pi r} \right)^2\nonumber\\
\hphantom{\frac{1}{(2 \pi r)^2}\frac{\vartheta_1''\big(c_1^{\RN} x/2 \pi r; \tau^{\RN}(s)\big)}{\vartheta_1\big(c_1^{\RN} x/2\pi r; \tau^{\RN}(s)\big)}=}{}
- \wp_{\cN^{\RN}}\big(t_{\ast}-s, c_1^{\RN} x\big)-\frac{\eta_{\cN^{\RN}}^1(t_{\ast}-s)}{\pi r}.\label{eqn:A_B2}
\end{gather}
Using (\ref{eqn:A_B1}) and (\ref{eqn:A_B2}), the l.h.s.\ of (\ref{eqn:eqA2}) can be written using $\zeta_{\cN^{\RN}}, \wp_{\cN^{\RN}}$ and $\eta_{\cN^{\RN}}^1$. Moreover, from the functional equation~(\ref{eqn:zeta_wp}) given in Appendix~\ref{sec:appendixA_5}, we can derive the following equalities
\begin{gather}
\sum_{\substack{1 \leq j, k, \ell \leq N, \cr j \not= k \not= \ell \not=j}}
(\zeta_{\cN^{\RN}}(t_{\ast}-s, x_j-x_k)+ \zeta_{\cN^{\RN}}(t_{\ast}-s, x_j+x_k))\nonumber\\
\qquad\quad{} \times (\zeta_{\cN^{\RN}}(t_{\ast}-s, x_j-x_{\ell})+ \zeta_{\cN^{\RN}}(t_{\ast}-s, x_j+x_{\ell}))\nonumber\\
\qquad{} =(N-2) \sum_{\substack{1 \leq j, k \leq N, \cr j \not= k}}
\big( \zeta_{\cN^{\RN}}(t_{\ast}-s, x_j-x_k)^2 + \zeta_{\cN^{\RN}}(t_{\ast}-s, x_j+x_k)^2 \big)\nonumber\\
 \qquad\quad{} - (N-2) \sum_{\substack{1 \leq j, k \leq N, \cr j \not= k}}
( \wp_{\cN^{\RN}}(t_{\ast}-s, x_j-x_k) + \wp_{\cN^{\RN}}(t_{\ast}-s, x_j+x_k) ),\label{eqn:identities1}\\
\sum_{\substack{1 \leq j, k \leq N, \cr j \not= k}}
\zeta_{\cN^{\RN}}(t_{\ast}-s, x_j) (\zeta_{\cN^{\RN}}(t_{\ast}-s, x_j -x_k)+ \zeta_{\cN^{\RN}}(t_{\ast}-s, x_j + x_k))\nonumber\\
\qquad{} = \frac{1}{4} \sum_{\substack{1 \leq j, k \leq N, \cr j \not= k}}
\big(\zeta_{\cN^{\RN}}(t_{\ast}-s, x_j - x_k)^2+\zeta_{\cN^{\RN}}(t_{\ast}-s, x_j + x_k)^2\big)\nonumber\\
\qquad\quad{} - \frac{1}{4} \sum_{\substack{1 \leq j, k \leq N, \cr j \not= k}}
(\wp_{\cN^{\RN}}(t_{\ast}-s, x_j - x_k) + \wp_{\cN^{\RN}}(t_{\ast}-s, x_j + x_k))\nonumber\\
 \qquad\quad{} + (N-1) \sum_{j=1}^N \zeta_{\cN^{\RN}}(t_{\ast}-s, x_j)^2- (N-1) \sum_{j=1}^N \wp_{\cN^{\RN}}(t_{\ast}-s, x_j),\label{eqn:identities1b}
\end{gather}
and
\begin{gather}
\sum_{\substack{1 \leq j, k \leq N, \cr j \not=k}} \zeta_{\cN^{\RN}}(t_{\ast}-s, 2 x_j) (\zeta_{\cN^{\RN}}(t_{\ast}-s, x_j -x_k)+ \zeta_{\cN^{\RN}}(t_{\ast}-s, x_j + x_k))\nonumber\\
 \qquad {} = \frac{1}{2} \sum_{\substack{1 \leq j, k \leq N, \cr j \not= k}}
\big(\zeta_{\cN^{\RN}}(t_{\ast}-s, x_j - x_k)^2+\zeta_{\cN^{\RN}}(t_{\ast}-s, x_j + x_k)^2\big)\nonumber\\
\qquad\quad{} - \frac{1}{2} \sum_{\substack{1 \leq j, k \leq N, \cr j \not= k}}
(\wp_{\cN^{\RN}}(t_{\ast}-s, x_j - x_k) + \wp_{\cN^{\RN}}(t_{\ast}-s, x_j + x_k))\nonumber\\
\qquad \quad{} + \frac{N-1}{2} \sum_{j=1}^N \zeta_{\cN^{\RN}}(t_{\ast}-s, 2 x_j)^2- \frac{N-1}{2} \sum_{j=1}^N \wp_{\cN^{\RN}}(t_{\ast}-s, 2 x_j).
\label{eqn:identities1c}
\end{gather}
Since $\eta_{\cN^{\RN}}(t_{\ast}-s, -x)=-\eta_{\cN^{\RN}}(t_{\ast}-s, x)$, we have
\begin{gather}
\sum_{\substack{1 \leq j, k, \ell \leq N, \cr j \not= k \not= \ell \not=j}}
(\zeta_{\cN^{\RN}}(t_{\ast}-s, x_j-x_k)+ \zeta_{\cN^{\RN}}(t_{\ast}-s, x_j+x_k))\{ (x_j-x_k)+(x_j+x_{\ell}) \}\nonumber\\
{} =(N-2) \sum_{\substack{1 \leq j, k \leq N, \cr j \not= k}}
\{ \zeta_{\cN^{\RN}}(t_{\ast}-s, x_j-x_k)(x_j-x_k) + \zeta_{\cN^{\RN}}(t_{\ast}-s, x_j+x_k)(x_j+x_k) \},\!\!\!\!\label{eqn:identities2}
\end{gather}
and
\begin{gather}
\sum_{\substack{1 \leq j, k, \ell \leq N, \cr j \not= k \not= \ell \not=j}}\{ (x_j-x_k) (x_j-x_{\ell}) + (x_j+x_k) (x_j+x_{\ell}) \}\nonumber\\
\qquad{} = (N-2) \sum_{\substack{1 \leq j, k \leq N, \cr j \not= k}} \big\{ (x_j-x_k)^2 + (x_j+x_k)^2 \big\}
= 4 (N-1) (N-2) \sum_{j=1}^N x_j^2. \label{eqn:identities3}
\end{gather}

Using the above formulas and the values of $\cN^{\RN}$ given by (\ref{eqn:cN_R}), we can show that (\ref{eqn:eqA2}) is reduced to the following simple equations
\begin{gather}
\frac{d \log g^{\BN}(s)}{ds}= - N(N-1)(2N-1) \frac{\eta_{2N-1}^1(t_{\ast}-s)}{2 \pi r},\nonumber\\
\frac{d \log g^{\CN}(s)}{ds}= - N(N^2-1) \frac{\eta_{2(N+1)}^1(t_{\ast}-s)}{\pi r},\nonumber\\
\frac{d \log g^{\DN}(s)}{ds}= - N(N-1)(N-2) \frac{\eta_{2(N-1)}^1(t_{\ast}-s)}{\pi r}.\label{eqn:g_RN_1}
\end{gather}
They give the conditions for $g^{\RN}(s)$ so that (\ref{eqn:eqA1}) holds. Since (\ref{eqn:Dedekind_eta}) gives for $\RN=\BN, \CN, \DN$
\begin{gather*}
\frac{d \log \eta\big(\tau^{\BN}(s)\big)}{ds}= (2N-1) \frac{\eta_{2N-1}^1(t_{\ast}-s)}{2 \pi r},\\
\frac{d \log \eta\big(\tau^{\CN}(s)\big)}{ds}= (N+1) \frac{\eta_{2(N+1)}^1(t_{\ast}-s)}{\pi r},\\
\frac{d \log \eta\big(\tau^{\DN}(s)\big)}{ds}= (N-1) \frac{\eta_{2(N-1)}^1(t_{\ast}-s)}{\pi r},
\end{gather*}
we f\/ind that
\begin{gather*}
g^{\BN}(s) = c^{\BN} \eta\big(\tau^{\BN}(s)\big)^{-N(N-1)},\nonumber\\
g^{\CN}(s) = c^{\CN} \eta\big(\tau^{\CN}(s)\big)^{-N(N-1)},\nonumber\\
g^{\DN}(s) = c^{\DN} \eta\big(\tau^{\DN}(s)\big)^{-N(N-2)},\nonumber
\end{gather*}
with constants $c^{\RN}$, $\RN=\BN, \CN, \DN$. By (\ref{eqn:qN3}), the conditions (\ref{eqn:Kolm_initial}) determine the constants.

In conclusion, under the condition (\ref{eqn:Kolm_initial}),
\begin{gather*}
 p^{\BN}(t, \y|s, \x)= \left( \frac{\eta\big(\tau^{\BN}(t)\big)}{\eta\big(\tau^{\BN}(s)\big)} \right)^{-N(N-1)}
\prod_{\ell=1}^N \frac{\vartheta_1\big(y_{\ell}/2 \pi r; \tau^{\BN}(t)\big)}{\vartheta_1\big(x_{\ell}/2 \pi r; \tau^{\BN}(s)\big)}\\
 \qquad{} \times \prod_{1 \leq j < k \leq N} \frac{\vartheta_1\big((y_k-y_j)/2 \pi r; \tau^{\BN}(t)\big)}
{\vartheta_1\big((x_k-x_j)/2 \pi r; \tau^{\BN}(s)\big)}\frac{\vartheta_1\big((y_k+y_j)/2 \pi r; \tau^{\BN}(t)\big)}
{\vartheta_1\big((x_k+x_j)/2 \pi r; \tau^{\BN}(s)\big)}q^{[0, \pi r]}_{\rm ar}(t-s, \y|\x)
\end{gather*}
solves (\ref{eqn:Kolm_B1}),
\begin{gather*} p^{\CN}(t, \y|s, \x) = \left( \frac{\eta\big(\tau^{\CN}(t)\big)}{\eta\big(\tau^{\CN}(s)\big)} \right)^{-N(N-1)}
\prod_{\ell=1}^N \frac{\vartheta_1\big(y_{\ell}/\pi r; \tau^{\CN}(t)\big)}{\vartheta_1\big(x_{\ell}/\pi r; \tau^{\CN}(s)\big)} \\
\qquad{}\times \prod_{1 \leq j < k \leq N} \frac{\vartheta_1\big((y_k-y_j)/2 \pi r; \tau^{\CN}(t)\big)} {\vartheta_1\big((x_k-x_j)/2 \pi r; \tau^{\CN}(s)\big)}
\frac{\vartheta_1\big((y_k+y_j)/2 \pi r; \tau^{\CN}(t)\big)} {\vartheta_1\big((x_k+x_j)/2 \pi r; \tau^{\CN}(s)\big)} q^{[0, \pi r]}_{\rm aa}(t-s, \y|\x)
\end{gather*}
solves (\ref{eqn:Kolm_C1}), and
\begin{gather*}
 p^{\DN}(t, \y|s, \x) = \left( \frac{\eta\big(\tau^{\DN}(t)\big)}{\eta\big(\tau^{\DN}(s)\big)} \right)^{-N(N-2)}\\
 \qquad {} \times \prod_{1 \leq j < k \leq N} \frac{\vartheta_1\big((y_k-y_j)/2 \pi r; \tau^{\DN}(t)\big)}
{\vartheta_1\big((x_k-x_j)/2 \pi r; \tau^{\DN}(s)\big)} \frac{\vartheta_1\big((y_k+y_j)/2 \pi r; \tau^{\DN}(t)\big)}
{\vartheta_1\big((x_k+x_j)/2 \pi r; \tau^{\DN}(s)\big)} q^{[0, \pi r]}_{\rm rr}(t-s, \y|\x)
\end{gather*}
solves (\ref{eqn:Kolm_D1}). The proof is thus completed.
\end{proof}

Now we prove Theorem \ref{thm:main2}.
\begin{proof} \looseness=-1 It is obvious that the backward Kolmogorov equations (\ref{eqn:Kolm_B1})--(\ref{eqn:Kolm_D1}) correspond to the~sys\-tems of SDEs (\ref{eqn:SDE_B1})--(\ref{eqn:SDE_D1}), respectively \cite{Kat16_Springer, RY99}. Due to the behavior (\ref{eqn:A_boundary}) of $A^{2 \pi r}_{\cN}$, we can show that \cite{Kat16_Springer, RY99} particles in $[0, \pi r]$ following (\ref{eqn:SDE_B1}) do not arrive at~0, and those fol\-lo\-wing~(\ref{eqn:SDE_C1}) do~not arrive at~0 nor~$\pi r$ with probability one, and thus we do not need to impose any boundary condition at these endpoints of $[0, \pi r]$ for these systems of SDEs. Hence the proof is completed.
\end{proof}

\section[Relaxation to equilibrium processes in trigonometric Dyson models of types C and D]{Relaxation to equilibrium processes\\ in trigonometric Dyson models of types C and D} \label{sec:relax}

As a corollary of Theorems \ref{thm:main1} and \ref{thm:main2}, we obtain the following trigonometric determinantal processes by taking the temporally homogeneous limit $t_{\ast} \to \infty$.

\begin{Corollary} \label{thm:tDYS_C_D} Assume that $\u \in \W^{(0, \pi r)}$. Then the trigonometric Dyson models of types~C and~D given by~\eqref{eqn:SDE_C2} and~\eqref{eqn:SDE_D2}, respectively, are determinantal with the spatio-temporal correlation kernels
\begin{gather}
\check{\bK}^{\CN}(s,x;t,y)= \sum_{j=1}^N p^{[0, \pi r]}_{\rm aa}(s, x|u_j) \check{M}^{\CN}_{\u, u_j}(t, y)-{\bf 1}(s>t) p^{[0, \pi r]}_{\rm aa}(s-t, x|y),\nonumber\\
\check{\bK}^{\DN}(s,x;t,y)= \sum_{j=1}^N p^{[0, \pi r]}_{\rm rr}(s, x|u_j) \check{M}^{\DN}_{\u, u_j}(t, y)-{\bf 1}(s>t) p^{[0, \pi r]}_{\rm rr}(s-t, x|y),\label{eqn:K_check_C_D}
\end{gather}
$(s, x), (t, y) \in [0, \infty) \times [0, \pi r]$, where
\begin{gather*}
\check{M}^{\RN}_{\u, u_j}(t, x)= \widetilde{\rE}\big[ \check{\Phi}^{\RN}_{\u, u_j}\big(x+i \widetilde{B}(t)\big)\big],\qquad \RN=\CN, \DN
\end{gather*}
with
\begin{gather}
\check{\Phi}^{\CN}_{\u, u_j}(z)= \frac{\sin(z/r)}{\sin(u_j/r)}
\prod_{\substack{1 \leq \ell \leq N, \cr \ell \not=j}}\frac{\sin((z-u_{\ell})/2r)}{\sin((u_j-u_{\ell})/2r)}
\frac{\sin((z+u_{\ell})/2r)}{\sin((u_j+u_{\ell})/2r)},\nonumber\\
\check{\Phi}^{\DN}_{\u, u_j}(z)=\prod_{\substack{1 \leq \ell \leq N, \cr \ell \not=j}}\frac{\sin((z-u_{\ell})/2r)}{\sin((u_j-u_{\ell})/2r)}\frac{\sin((z+u_{\ell})/2r)}{\sin((u_j+u_{\ell})/2r)},
\qquad j =1,2, \dots, N.\label{eqn:Phi_check_C_D}
\end{gather}
\end{Corollary}

We f\/ind that in the limit $t_{\ast} \to \infty$, the equality (\ref{eqn:factor1b}) with $\RN=\CN$ and $\DN$ gives
\begin{gather}
\det_{1 \leq j, k \leq N} \big[ \check{f}^{\CN}_j(u_k)\big]= (-1)^{N(N-1)/2} 2^{N(N-1)}\prod_{\ell=1}^N \sin \left( \frac{u_{\ell}}{r} \right)\nonumber\\
\hphantom{\det_{1 \leq j, k \leq N} \big[ \check{f}^{\CN}_j(u_k)\big]=}{} \times\prod_{1 \leq j < k \leq N}\sin \left( \frac{u_k-u_j}{2r} \right)\sin \left( \frac{u_k+u_j}{2r} \right),\nonumber\\
\det_{1 \leq j, k \leq N} \big[ \check{f}^{\DN}_j(u_k)\big]= (-1)^{N(N-1)/2} 2^{(N-1)^2}
\prod_{1 \leq j < k \leq N}\sin \left( \frac{u_k-u_j}{2r} \right)\sin \left( \frac{u_k+u_j}{2r} \right)\label{eqn:det_f_tri}
\end{gather}
with
\begin{gather}
\check{f}_j^{\CN}(z)=\sin \left( \frac{jz}{r}\right),\qquad\check{f}_j^{\DN}(z)= \cos \left( \frac{(j-1)z}{r} \right), \qquad j \in \Z. \label{eqn:f:tri}
\end{gather}
By Lemma \ref{thm:det_lemma_1}, we see that (\ref{eqn:Phi_check_C_D}) can be expanded as
\begin{gather*}
\check{\Phi}^{\RN}_{\u, u_j}(z)=\sum_{k=1}^N \check{\phi}^{\RN}_{\u, u_j}(k) \check{f}^{\RN}_k(z), \qquad \RN=\CN, \DN,
\end{gather*}
and
\begin{gather}
\sum_{\ell=1}^N \check{f}^{\RN}_j(u_{\ell}) \check{\phi}^{\RN}_{\u, u_{\ell}}(k)=\delta_{j k},\qquad j, k \in \{1,2, \dots, N \}.\label{eqn:relax1}
\end{gather}
We note that the transition probability densities (\ref{eqn:p_abs}) and (\ref{eqn:p_ref}) are written as~(\ref{eqn:p_abs_A}) \linebreak and~(\ref{eqn:p_ref_A}), respectively. Hence we have
\begin{gather}
p^{[0, \pi r]}_{\rm aa}(t, y|x)= \frac{1}{\pi r} \sum_{n \in \Z} e^{-n^2 t/2r^2}\check{f}^{\CN}_n(y) \check{f}^{\CN}_n(x)
= \frac{2}{\pi r} \sum_{\ell=1}^{\infty} e^{-\ell^2 t/2r^2}\check{f}^{\CN}_{\ell}(y) \check{f}^{\CN}_{\ell}(x),\nonumber\\
p^{[0, \pi r]}_{\rm rr}(t, y|x)= \frac{1}{\pi r}\left( \check{f}^{\DN}_1(y) \check{f}^{\DN}_1(x)
+ 2 \sum_{\ell=2}^{\infty} e^{-(\ell-1)^2 t/2 r^2}\check{f}^{\DN}_{\ell}(y) \check{f}^{\DN}_{\ell}(x) \right).\label{eqn:p_tri}
\end{gather}
Now we show relaxation processes to equilibria, which are typical non-equilibrium phenomena.

\begin{Proposition}\label{thm:relax_tDYS_C_D} For any initial configuration $\u \in \W^{(0, \pi r)}$, the trigonometric Dyson models of types C and D
exhibit relaxation to the equilibrium processes. The equilibrium processes are determinantal with the correlation kernels
\begin{gather}
\check{\bK}^{\CN}_{\rm eq}(t-s, x, y)\nonumber\\
\qquad {}= \begin{cases}
\displaystyle \frac{1}{\pi r}\sum_{n \colon |n| \leq N}e^{n^2(t-s)/2r^2} \sin \left( \frac{nx}{r} \right)\sin \left( \frac{ny}{r} \right),
& \text{if $t > s$},\vspace{1mm}\\
\displaystyle \frac{1}{2 \pi r} \left[\frac{\sin\{(2N+1)(y-x)/2r\}}{\sin\{(y-x)/2r \}}- \frac{\sin\{(2N+1)(y+x)/2r\}}{\sin\{(y+x)/2r \}} \right],
& \text{if $t=s$},\vspace{1mm}\\
\displaystyle - \frac{1}{\pi r}\sum_{n \colon |n| \geq N+1}e^{n^2(t-s)/2r^2} \sin \left( \frac{nx}{r} \right)\sin \left( \frac{ny}{r} \right),
& \text{if $t < s$},
\end{cases}\label{eqn:K_eq_C}
\end{gather}
and
\begin{gather}
\check{\bK}^{\DN}_{\rm eq}(t-s, x, y)\nonumber\\
\qquad {}= \begin{cases}
\displaystyle \frac{1}{\pi r}\sum_{n \colon |n| \leq N-1}e^{n^2(t-s)/2r^2} \cos \left( \frac{nx}{r} \right)\cos \left( \frac{ny}{r} \right),
& \text{if $t > s$},\vspace{1mm}\\
\displaystyle \frac{1}{2 \pi r} \left[\frac{\sin\{(2N-1)(y-x)/2r\}}{\sin\{(y-x)/2r \}}+ \frac{\sin\{(2N-1)(y+x)/2r\}}{\sin\{(y+x)/2r \}} \right],
& \text{if $t=s$},\vspace{1mm}\\
\displaystyle - \frac{1}{\pi r}\sum_{n \colon |n| \geq N} e^{n^2(t-s)/2r^2} \cos \left( \frac{nx}{r} \right)\cos \left( \frac{ny}{r} \right),
& \text{if $t < s$},
\end{cases}\label{eqn:K_eq_D}
\end{gather}
respectively.
\end{Proposition}

\begin{proof} Here we prove the convergence of correlation kernels. It implies the convergence of the generating functions of spatio-temporal correlation functions (i.e., the Laplace transformations of multi-time joint probability densities). Hence the convergence of process is concluded in the sense of f\/inite-dimensional distributions~\cite{KT10}. Assume $\u \in \W^{(0, \pi r)}$ and let
\begin{gather*}
\g^{\CN}_{\u}(s, x; t, y)= \sum_{j=1}^N p^{[0,\pi r]}_{\rm aa}(s, x | u_j) \check{M}^{\CN}_{\u, u_j}(t, y),\\
\g^{\DN}_{\u}(s, x; t, y)= \sum_{j=1}^N p^{[0,\pi r]}_{\rm rr}(s, x | u_j)\check{M}^{\DN}_{\u, u_j}(t, y),\qquad (s,x),(t, y) \in [0,\infty) \times [0,\pi r].
\end{gather*}
It is easy to verify that
\begin{gather*}
\widetilde{\rE} \big[\check{f}^{\CN}_k\big(x+i \widetilde{B}(t)\big)\big]= e^{k^2 t/2 r^2} \check{f}^{\CN}_k(x),\\
\widetilde{\rE} \big[\check{f}^{\DN}_k\big(x+i \widetilde{B}(t)\big)\big] = e^{(k-1)^2 t/2 r^2} \check{f}^{\DN}_k(x), \qquad k \in \Z,
\end{gather*}
and thus
\begin{gather*}
\check{M}^{\CN}_{\u, u_j}(t, y)= \sum_{k=1}^N \check{\phi}^{\CN}_{\u, u_j}(k) e^{k^2 t/2r^2}\check{f}^{\CN}_k(y),\\
\check{M}^{\DN}_{\u, u_j}(t, y)= \sum_{k=1}^N \check{\phi}^{\DN}_{\u, u_j}(k) e^{(k-1)^2 t/2r^2}
\check{f}^{\DN}_k(y), \qquad j \in \{1,2, \dots, N\}.
\end{gather*}
Then we have
\begin{gather*}
\g^{\RN}_{\u}(s, x; t, y) = \g^{\RN}_{\rm eq}(t-s, x, y)+\r^{\RN}_{\u}(s,x;t,y),\qquad \RN=\CN, \DN
\end{gather*}
with
\begin{gather}
\g^{\CN}_{\rm eq}(t-s, x, y)= \frac{2}{\pi r} \sum_{\ell=1}^N e^{\ell^2(t-s)/2r^2}\check{f}^{\CN}_{\ell}(x) \check{f}^{\CN}_{\ell}(y)\nonumber\\
\hphantom{\g^{\CN}_{\rm eq}(t-s, x, y)}{} = \frac{1}{\pi r} \sum_{n\colon |n| \leq N}
e^{n^2(t-s)/2r^2} \sin \left( \frac{nx}{r} \right) \sin \left( \frac{ny}{r} \right),\nonumber\\
\r^{\CN}_{\u}(s,x;t,y) = \frac{2}{\pi r} \sum_{k=1}^N e^{k^2(t-s)/2r^2} \check{f}^{\CN}_k(y)
\sum_{\ell=N+1}^{\infty} e^{-(\ell^2-k^2) s/2r^2} \check{f}^{\CN}_{\ell}(x) \nonumber\\
\hphantom{\r^{\CN}_{\u}(s,x;t,y) =}{} \times \sum_{j=1}^N \check{f}^{\CN}_{\ell}(u_j) \check{\phi}^{\CN}_{\u, u_j}(k),\label{eqn:rC1}
\end{gather}
and
\begin{gather}
\g^{\DN}_{\rm eq}(t-s, x, y) = \frac{1}{\pi r} \left(\check{f}^{\DN}_1(x) \check{f}^{\DN}_1(y)
+ 2 \sum_{\ell=2}^N e^{(\ell-1)^2(t-s)/2r^2}\check{f}^{\DN}_{\ell}(x) \check{f}^{\DN}_{\ell}(y) \right)\nonumber\\
\hphantom{\g^{\DN}_{\rm eq}(t-s, x, y) }{} = \frac{1}{\pi r} \sum_{n: |n| \leq N-1}
e^{n^2(t-s)/2r^2} \cos \left( \frac{nx}{r} \right) \cos \left( \frac{ny}{r} \right),\nonumber\\
\r^{\DN}_{\u}(s,x;t,y)= \frac{2}{\pi r} \sum_{k=1}^N e^{(k-1)^2(t-s)/2r^2}\check{f}^{\DN}_k(y)\nonumber\\
\hphantom{\r^{\DN}_{\u}(s,x;t,y)=}{} \times
\sum_{\ell = N+1}^{\infty} e^{-\{(\ell-1)^2-(k-1)^2\} s/2r^2}
\check{f}^{\DN}_{\ell}(x) \sum_{j=1}^N \check{f}^{\DN}_{\ell}(u_j) \check{\phi}^{\DN}_{\u, u_j}(k),\label{eqn:rD1}
\end{gather}
where (\ref{eqn:relax1}) and (\ref{eqn:p_tri}) were used. For any f\/ixed $s, t \in [0, \infty)$, it is obvious that
\begin{gather*}
\lim_{T \to \infty} \r^{\RN}_{\u}(s+T,x; t+T, y)=0,\qquad \RN=\CN, \DN,
\end{gather*}
uniformly on any subset of $(x, y) \in (0, \pi r)^2$, since the summations of $k$ are taken for $1 \leq k \leq N$, while those of $\ell$ are taken for $\ell \geq N+1$ in (\ref{eqn:rC1}) and (\ref{eqn:rD1}). Hence
\begin{gather*}
 \lim_{T \to \infty} \check{\bK}^{\CN}_{\u}(s+T, x; t+T, y) = \g^{\CN}_{\rm eq}(t-s, x,y)-{\bf 1}(s>t) p^{[0, \pi r]}_{\rm aa}(s-t, x|y), \\
\lim_{T \to \infty} \check{\bK}^{\DN}_{\u}(s+T, x; t+T, y) = \g^{\DN}_{\rm eq}(t-s, x,y) -{\bf 1}(s>t) p^{[0, \pi r]}_{\rm rr}(s-t, x|y),
\end{gather*}
in the same sense. It is easy to conf\/irm that the limit kernels are represented by (\ref{eqn:K_eq_C}) and (\ref{eqn:K_eq_D}), if we use (\ref{eqn:p_tri}). The limit kernels (\ref{eqn:K_eq_C}) and (\ref{eqn:K_eq_D}) depend on time dif\/ference $t-s$, which implies that the determinantal processes def\/ined by them are temporally homogeneous. The determinantal processes with the spatio-temporal correlation kernels (\ref{eqn:K_eq_C}) and (\ref{eqn:K_eq_D}) are equilibrium processes, which are reversible with respect to the {\it determinantal point processes} with the spatial correlation kernels
\begin{gather}
\check{\rm K}_{\rm eq}^{\CN}(x,y) = \frac{1}{2 \pi r} \left[\frac{\sin\{(2N+1)(y-x)/2r\}}{\sin\{(y-x)/2r \}}- \frac{\sin\{(2N+1)(y+x)/2r\}}{\sin\{(y+x)/2r \}} \right],\nonumber\\
\check{\rm K}_{\rm eq}^{\DN}(x,y)= \frac{1}{2 \pi r} \left[\frac{\sin\{(2N-1)(y-x)/2r\}}{\sin\{(y-x)/2r \}}
+ \frac{\sin\{(2N-1)(y+x)/2r\}}{\sin\{(y+x)/2r \}} \right],\label{eqn:dpp}
\end{gather}
$(x, y) \in [0, \pi r]^2$, respectively. The convergence of processes is irreversible. Thus all statements of the present proposition have been proved.
\end{proof}

We note that
\begin{gather*}
\check{\rho}^{\CN}_{\rm eq}(x) = \lim_{y \to x} \check{\rm K}^{\CN}_{\rm eq}(x, y)=\frac{2}{\pi r} \sum_{n=1}^N \sin^2 \left(\frac{nx}{r} \right),\\
\check{\rho}^{\DN}_{\rm eq}(x)= \lim_{y \to x} \check{\rm K}^{\DN}_{\rm eq}(x, y)=\frac{1}{\pi r} \left\{ 1 + 2 \sum_{n=1}^{N-1} \cos^2 \left(\frac{nx}{r} \right) \right\},
\end{gather*}
and hence
\begin{gather*}
\int_0^{\pi r} \check{\rho}^{\CN}_{\rm eq}(x) dx = \int_0^{\pi r} \check{\rho}^{\DN}_{\rm eq}(x) dx =N,
\end{gather*}
as required.

\section{Concluding remarks and open problems} \label{sec:remarks}

In the previous paper \cite{Kat15,Kat16} and in this paper (Theorem~\ref{thm:main1}), we have introduced seven families of interacting particle systems
$\Xi^{\RN}(t)=\sum\limits_{j=1}^N \delta_{X^{\RN}_j(t)}$, $t \in [0, t_{\ast})$ governed by the probability laws~$\P^{\RN}_{\u}$ associated with the irreducible reduced af\/f\/ine root systems denoted by $\RN=\AN, \BN, \BNv, \CN$, $\CNv, \BCN, \DN$. When we proved that they are determinantal processes, we showed that without change of expectations of symmetric functions of $\big\{X^{\RN}_j(\cdot)\big\}_{j=1}^N$, $\P^{\RN}_{\u}$ can be replaced by the signed measures $\widehat{\P}^{\RN}_{\u}$. Def\/ine $\widehat{\cF}^{\RN}_t= \sigma\big( \widehat{\Xi}^{\RN}(s) \colon 0 \leq s \leq t\big)$, $t \in [0, t_{\ast})$. As the simplest corollary of this fact, we can conclude that at any time $0 < t < t_{\ast}$,
\begin{gather*}
\widehat{\E}^{\RN}_{\u} \big[ {\bf 1}\big( \widehat{\Xi}^{\RN}(t) \in \widehat{\cF}^{\RN}_t\big) \big] =1
\qquad \text{for} \quad \RN=\AN, \BN, \BNv, \CN, \CNv, \BCN, \DN.
\end{gather*}
This is nothing but a rather trivial statement such that the processes $\big(\widehat{\Xi}^{\RN}(t)\big)_{t \in [0, t_{\ast})}$ are well-normalized, but it provides nontrivial multiple-integral equalities including parameters $t \in [0, t_{\ast})$ and $\u=(u_1, \dots, u_N)$. For example, for $\RN=\DN$, we will have
\begin{gather}
\int_{[0, \pi r]^N} d \x \prod_{\ell=1}^N \frac{1}{2 \pi r}\left\{ \vartheta_3 \left( \frac{x_{\ell}-u_{\ell}}{2 \pi r}; \frac{i t}{2 \pi r^2} \right)
+ \vartheta_3 \left( \frac{x_{\ell}+u_{\ell}}{2 \pi r}; \frac{i t}{2 \pi r^2} \right) \right\}\label{eqn:integral_D}\\
{} \times \prod_{1 \leq j < k \leq N}\!\!
\frac{\vartheta_1((x_k-x_j)/2 \pi r; \tau^{\DN}(t))}{\vartheta_1((u_k-u_j)/2 \pi r; \tau^{\DN}(0))}
\frac{\vartheta_1((x_k+x_j)/2 \pi r; \tau^{\DN}(t))}{\vartheta_1((u_k+u_j)/2 \pi r; \tau^{\DN}(0))}
= \left( \frac{\eta(\tau^{\DN}(t))}{\eta(\tau^{\DN}(0))} \right)^{N(N-2)}\!\!,\nonumber
\end{gather}
where $\eta(\tau)$ is the Dedekind modular function (\ref{eqn:Dedekind1}) and $\tau^{\DN}(t)=i (N-1)(t_{\ast}-t)/\pi r^2$.

In the previous papers \cite{Kat15,Kat16} and in Theorem \ref{thm:main2} we have identif\/ied the systems of SDEs which are solved by the four families of determinantal processes, $\big(\big(\Xi^{\AN}(t)\big)_{t \in [0, t_{\ast})}, \P^{\AN}_{\u}\big)$, $\u \in \W_N^{(0, 2 \pi r)}$, and $\big(\big(\Xi^{\RN}(t)\big)_{t \in [0, t_{\ast})}, \P^{\RN}_{\u}\big)$, $\u \in \W_N^{(0, \pi r)}$ with $\RN=\BN, \CN$ and $\DN$. The systems of SDEs for other cases $\RN=\BNv, \CNv, \BCN$ are not yet clarif\/ied. The exceptional cases of reduced irreducible af\/f\/ine root systems and the non-reduced irreducible af\/f\/ine root systems~\cite{Mac72} will be studied from the view point of the present stochastic analysis.

The classical Dyson models of type A in $\R$ given by (\ref{eqn:SDE_A3}) and of types C and D in $[0, \infty)$ given by (\ref{eqn:SDE_C3}) and~(\ref{eqn:SDE_D3}) are realized as the eigenvalue processes of Hermitian-matrix-valued Brownian motions with specif\/ied symmetry \cite{Dys62,KT04}. It will be an interesting problem to construct the matrix-valued Brownian motions such that the eigenvalue processes of them provide the elliptic Dyson models; (\ref{eqn:SDE_A1}) with $\beta=2$ and (\ref{eqn:SDE_B1})--(\ref{eqn:SDE_D1}).

As shown in Section \ref{sec:proof2}, the factors represented by the Dedekind modular function (\ref{eqn:Dedekind1}) in the determinantal martingale-functions
(\ref{eqn:D_Ra})--(\ref{eqn:D_Rc}) are essential in the proof of Theorem \ref{thm:main2}. We have noted that $\cN^{\RN}$ given by (\ref{eqn:cN_R}) is identif\/ied with the quantity g given for the reduced irreducible af\/f\/ine root systems in \cite[Appendix~1]{Mac72}. Interpretation of the formulas~(\ref{eqn:c_0}) from the view point of representation theory is desired.

As mentioned after Theorem \ref{thm:main2} in Section~\ref{sec:Introduction}, if we take the temporally homogeneous limit $t_{\ast} \to \infty$, the elliptic Dyson models of types A, B, C, and D become the corresponding trigonometric Dyson models, and in the further limit $r \to \infty$, they are reduced to the Dyson models in~$\R$, in $[0, \infty)$ with an absorbing boundary condition at the origin, and in $[0, \infty)$ with a ref\/lecting boundary condition at the origin, respectively. In the limit $r \to \infty$, $p^{[0, \pi r]}_{\rm aa}$ and $p^{[0, \pi r]}_{\rm rr}$ given by~(\ref{eqn:p_abs}) and~(\ref{eqn:p_ref}) become
\begin{gather*}
p^{[0, \infty)}_{\rm aa}(t, y|x)= p_{\rm BM}(t, y|x)-p_{\rm BM}(t, y|-x),\\
p^{[0, \infty)}_{\rm rr}(t, y|x)= p_{\rm BM}(t, y|x)+p_{\rm BM}(t, y|-x),
\end{gather*}
and in the double limit $t_{\ast} \to \infty$, $r \to \infty$, $\Phi^{\CN}_{\u, u_j}(z)$ and $\Phi^{\DN}_{\u, u_j}(z)$ given by (\ref{eqn:Phi_Ra})--(\ref{eqn:Phi_Rc}) become
\begin{gather*}
\Phi^{\CN}_{\u, u_j}(z) \to \frac{z_j}{u_j} \Phi^{[0, \infty)}_{\u, u_j}(z), \qquad \Phi^{\DN}_{\u, u_j}(z) \to \Phi^{[0, \infty)}_{\u, u_j}(z)
\end{gather*}
with
\begin{gather*}
\Phi^{[0, \infty)}_{\u, u_j}(z)=\prod_{\substack{1 \leq \ell \leq N, \cr \ell \not=j}} \frac{z^2-u_{\ell}^2}{u_j^2-u_{\ell}^2}.
\end{gather*}
Then the spatio-temporal correlation kernels (\ref{eqn:K1}) are reduced to the following
\begin{gather*}
\bK^{\CN}_{\u}(s, x; t, y) = \sum_{j=1}^N p^{[0, \infty)}_{\rm aa}(s, x|u_j)\widetilde{\rE} \left[\frac{y+i \widetilde{B}(t)}{u_j}\Phi^{[0, \infty)}_{\u, u_j}\big(y+i \widetilde{B}(t)\big) \right]\\
\hphantom{\bK^{\CN}_{\u}(s, x; t, y)=}{}-{\bf 1}(s>t) p^{[0, \infty)}_{\rm aa}(s-t, x|y),\nonumber\\
\bK^{\DN}_{\u}(s, x; t, y)= \sum_{j=1}^N p^{[0, \infty)}_{\rm rr}(s, x|u_j)\widetilde{\rE} \big[\Phi^{[0, \infty)}_{\u, u_j}\big(y+i \widetilde{B}(t)\big) \big]-{\bf 1}(s>t) p^{[0, \infty)}_{\rm rr}(s-t, x|y),
\end{gather*}
$(s, x), (t, y) \in [0, \infty) \times [0, \infty)$. Consider the $D$-dimensional Bessel processes BES$^{(D)}$ in $[0, \infty)$, whose transition probability densities $p^{{\rm BES}^{(D)}}$ are expressed using the modif\/ied Bessel func\-tion~$I_{\nu}(z)$ with $\nu=(D-2)/2$ \cite{Kat16_Springer, RY99}. In particular, we f\/ind that
\begin{gather*}
p^{{\rm BES}^{(3)}}(t, y|x)=\frac{y}{x} p^{[0, \infty)}_{\rm aa}(t, y|x),\qquad p^{{\rm BES}^{(1)}}(t, y|x)= p^{[0, \infty)}_{\rm rr}(t, y|x),
\end{gather*}
$x, y \in (0, \infty)$, $t \geq 0$. If we put
\begin{gather}
\bK^{{\rm BES}^{(3)}}_{\u}(s, x; t, y)= \sum_{j=1}^N p^{{\rm BES}^{(3)}}(s, x|u_j)
\widetilde{\rE} \left[\frac{y+i \widetilde{B}(t)}{y} \Phi^{[0, \infty)}_{\u, u_j}\big(y+i \widetilde{B}(t)\big) \right]\nonumber\\
\hphantom{\bK^{{\rm BES}^{(3)}}_{\u}(s, x; t, y)=}{}
- {\bf 1}(s>t) p^{{\rm BES}^{(3)}}(s-t, x|y),\label{eqn:K_BES3}\\
\bK^{{\rm BES}^{(1)}}_{\u}(s, x; t, y)= \sum_{j=1}^N p^{{\rm BES}^{(1)}}(s, x|u_j)\widetilde{\rE} \big[\Phi^{[0, \infty)}_{\u, u_j}\big(y+i \widetilde{B}(t)\big) \big]\nonumber\\
\hphantom{\bK^{{\rm BES}^{(1)}}_{\u}(s, x; t, y)=}{} - {\bf 1}(s>t) p^{{\rm BES}^{(1)}}(s-t, x|y),\label{eqn:K_BES1}
\end{gather}
the following equalities are established
\begin{gather}
\bK^{\CN}_{\u}(s, x; t, y) = \frac{y}{x} \bK^{{\rm BES}^{(3)}}_{\u}(s, x; t, y), \qquad
\bK^{\DN}_{\u}(s, x; t, y)= \bK^{{\rm BES}^{(1)}}_{\u}(s, x; t, y),\label{eqn:BES}
\end{gather}
$(s, x), (t, y) \in [0, \infty) \times [0, \infty)$. The functions (\ref{eqn:K_BES3}) and (\ref{eqn:K_BES1}) are the spatio-temporal correlation kernels of the $N$-particle systems of BES$^{(3)}$ and BES$^{(1)}$ with noncolliding condition (see \cite[Section~7]{Kat14}). It is easy to verify that the spatio-temporal Fredholm determinant~(\ref{eqn:Fredholm}) is invariant under the transformation of kernel
\begin{gather*}
\widehat{\K}_{\u}(s, x; t, y) \to \frac{a(t, y)}{a(s, x)} \widehat{\K}_{\u}(s, x; t, y)
\end{gather*}
with an arbitrary continuous function~$a$. Then (\ref{eqn:BES}) implies that the elliptic Dyson models of types C and D given by Theorem \ref{thm:main2} are deduced to the noncolliding BES$^{(3)}$ and BES$^{(1)}$ in the double limit $t_{\ast} \to \infty$, $r \to \infty$, respectively. In other words,
the present elliptic Dyson models of types C and D can be regarded as the elliptic extensions of the noncolliding BES$^{(3)}$ and BES$^{(1)}$, respectively.
Trigonometric and elliptic extensions of noncolliding BES$^{(D)}$ with general $D \geq 1$ \cite{Kat14, KT11} will be studied. Moreover, inf\/inite-particle limits of the elliptic determinantal processes should be studied \cite{Kat15,Kat16, KT10}.

Connection between the present elliptic determinantal processes and probabilistic discrete models with elliptic weights \cite{Betea11,BGR10,Sch12,Sch07,SY17} will be also an interesting future problem.

\appendix

\section{The Jacobi theta functions and related functions}\label{sec:appendixA}

\subsection{Notations and formulas of the Jacobi theta functions}\label{sec:appendixA_1}

Let
\begin{gather*}
z=e^{v \pi i}, \qquad q=e^{\tau \pi i},
\end{gather*}
where $v, \tau \in \C$ and $\Im \tau > 0$. The Jacobi theta functions are def\/ined as follows \cite{NIST10, WW27}
\begin{gather}
\vartheta_0(v; \tau) =-i e^{\pi i (v+\tau/4)} \vartheta_1 \left( v + \frac{\tau}{2}; \tau \right)=
\sum_{n \in \Z} (-1)^n q^{n^2} z^{2n} =1+ 2 \sum_{n=1}^{\infty}(-1)^n e^{\tau \pi i n^2} \cos(2 n \pi v),\nonumber\\
\vartheta_1(v; \tau) = i \sum_{n \in \Z} (-1)^n q^{(n-1/2)^2} z^{2n-1}
=2 \sum_{n=1}^{\infty} (-1)^{n-1} e^{\tau \pi i (n-1/2)^2} \sin\{(2n-1) \pi v\},\nonumber\\
\vartheta_2(v; \tau)= \vartheta_1 \left( v+ \frac{1}{2}; \tau \right)=\sum_{n \in \Z} q^{(n-1/2)^2} z^{2n-1}
=2 \sum_{n=1}^{\infty} e^{\tau \pi i (n-1/2)^2} \cos \{(2n-1) \pi v\},\nonumber\\
\vartheta_3(v; \tau) = e^{\pi i (v+\tau/4)} \vartheta_1 \left( v+\frac{1+\tau}{2}; \tau \right)= \sum_{n \in \Z} q^{n^2} z^{2n}=1 + 2 \sum_{n=1}^{\infty} e^{\tau \pi i n^2} \cos (2 n \pi v).\label{eqn:theta}
\end{gather}
(Note that the present functions $\vartheta_{\mu}(v; \tau)$, $\mu=1,2,3$ are denoted by $\vartheta_{\mu}(\pi v,q)$, and $\vartheta_0(v;\tau)$ by $\vartheta_4(\pi v,q)$ in~\cite{WW27}.) For $\Im \tau >0$, $\vartheta_{\mu}(v; \tau)$, $\mu=0,1,2,3$ are holomorphic for $|v| < \infty$ and satisfy the partial dif\/ferential equation
\begin{gather}
\frac{\partial \vartheta_{\mu}(v; \tau)}{\partial \tau}=
\frac{1}{4 \pi i} \frac{\partial^2 \vartheta_{\mu}(v; \tau)}{\partial v^2}.\label{eqn:Jacobi_eq}
\end{gather}
They have the quasi-periodicity; for instance, $\vartheta_1$ satisf\/ies
\begin{gather*}
\vartheta_1(v+1; \tau)=-\vartheta_1(v; \tau), \qquad \vartheta_1(v+\tau; \tau)=-e^{-\pi i (2v+\tau)} \vartheta_1(v; \tau).
\end{gather*}
By the def\/inition (\ref{eqn:theta}), when $\Im \tau > 0$,
\begin{gather}
\vartheta_1(0; \tau)=\vartheta_1(1; \tau)=0, \qquad \vartheta_1(x; \tau) > 0, \qquad x \in (0,1),\nonumber\\
\vartheta_0(x; \tau) > 0, \qquad x \in \R. \label{eqn:theta_positive}
\end{gather}
We see the asymptotics
\begin{gather}
\vartheta_0(v; \tau) \sim 1, \qquad \vartheta_1(v; \tau) \sim 2 e^{\tau \pi i/4} \sin (\pi v), \qquad
\vartheta_2(v; \tau) \sim 2 e^{\tau \pi i/4} \cos(\pi v),\nonumber\\
\vartheta_3(v; \tau) \sim 1\qquad \text{in} \quad \Im \tau \to + \infty \qquad (\text{i.e.}, \ q=e^{\tau \pi i} \to 0).\label{eqn:theta_asym}
\end{gather}
The following functional equalities are known as Jacobi's imaginary transformations \cite{NIST10, WW27}
\begin{gather}
\vartheta_0(v; \tau)= e^{\pi i/4} \tau^{-1/2} e^{-\pi i v^2/\tau}\vartheta_2 \left( \frac{v}{\tau}; - \frac{1}{\tau} \right),\nonumber\\
\vartheta_1(v; \tau)= e^{3 \pi i/4} \tau^{-1/2} e^{-\pi i v^2/\tau}\vartheta_1 \left( \frac{v}{\tau}; - \frac{1}{\tau} \right),\nonumber\\
\vartheta_3(v; \tau)= e^{\pi i/4} \tau^{-1/2} e^{-\pi i v^2/\tau}\vartheta_3 \left( \frac{v}{\tau}; - \frac{1}{\tau} \right). \label{eqn:Jacobi_imaginary}
\end{gather}

\subsection[Basic properties of $A_{\cN}^{2 \pi r}(t_{\ast}-t,x)$]{Basic properties of $\boldsymbol{A_{\cN}^{2 \pi r}(t_{\ast}-t,x)}$}\label{sec:appendixA_2}

For $0< t_{\ast} < \infty$, $0< r < \infty$, $0< \cN < \infty$, the function $A_{\cN}^{2 \pi r}(t_{\ast}-t,x)$ is def\/ined by~(\ref{eqn:A1}), which is written as
\begin{gather}
A_{\cN}^{2 \pi r}(t_{\ast}-t,x) = \frac{1}{2 \pi r}\frac{\vartheta_1'\big(x/2 \pi r ; i \cN (t_{\ast}-t)/2 \pi r^2\big)}
{\vartheta_1\big(x/2 \pi r ; i \cN (t_{\ast}-t)/2 \pi r^2\big)},\qquad t \in [0, t_{\ast}), \label{eqn:A1b}
\end{gather}
where $\vartheta_1'(v; \tau)=\partial \vartheta_1(v; \tau)/\partial v$. As a function of $x \in \R$, $A_{\cN}^{2 \pi r}(t_{\ast}-t,x)$ is odd,
\begin{gather}
A_{\cN}^{2 \pi r}(t_{\ast}-t,-x)=-A_{\cN}^{2 \pi r}(t_{\ast}-t, x),\label{eqn:A_odd}
\end{gather}
and periodic with period $2 \pi r$
\begin{gather*}
A_{\cN}^{2 \pi r}(t_{\ast}-t, x +2 m \pi r)=A_{\cN}^{2 \pi r}(t_{\ast}-t, x),\qquad m \in \Z.
\end{gather*}
It has only simple poles at $x= 2 m \pi r$, $m \in \Z$, and simple zeroes at $x=(2m+1) \pi r$, $m \in \Z$. Independently of the values of $t \in [0, t_{\ast})$ and $0< \cN < \infty$, $A_{\cN}^{2 \pi r}(t_{\ast}-t, x)$ behaves as
\begin{gather}
A_{\cN}^{2 \pi r}(t_{\ast}-t, x) \sim
\begin{cases}
\dfrac{1}{x}, & \text{as $x \downarrow 0 $}, \vspace{1mm}\\
-\dfrac{1}{2 \pi r-x}, & \text{as $x \uparrow 2 \pi r$}.
\end{cases}\label{eqn:A_boundary}
\end{gather}
Using Jacobi's imaginary transformation (\ref{eqn:Jacobi_imaginary}), we can show by (\ref{eqn:theta_asym}) that
\begin{gather}
A^{2 \pi r}_{\cN}(t_{\ast}-t, x) \sim
\begin{cases}
-\dfrac{x-\pi r}{\cN (t_{\ast}-t)}, & \text{if $x>0$}, \vspace{1mm}\\
-\dfrac{x+\pi r}{\cN (t_{\ast}-t)}, & \text{if $x<0$},
\end{cases}
\qquad \text{as $t \uparrow t_{\ast}$}.\label{eqn:A_t_ast}
\end{gather}

\subsection{Dedekind modular function}\label{sec:appendixA_3}

The Dedekind modular function $\eta(\tau)$ is def\/ined by (\ref{eqn:Dedekind1}), that is,
\begin{gather*}
\eta(\tau)=e^{\tau \pi i/12}\prod_{n=1}^{\infty} (1-e^{2 n \tau \pi i}),\qquad \Im \tau > 0.
\end{gather*}
For $t \in [0, t_{\ast})$, $0< \cN < \infty$, def\/ine
\begin{gather}
\eta_{\cN}^1(t_{\ast}-t)= \frac{\pi^2}{\omega_1} \left. \left( \frac{1}{12} - 2 \sum_{n=1}^{\infty}\frac{n q^{2n}}{1-q^{2n}} \right)\right|_{\omega_1=\pi r, \, q=e^{-\cN(t_{\ast}-t)/2 r^2}}.\label{eqn:eta1_1}
\end{gather}
Then the following equality is established \cite{Kat15,Kat16}
\begin{gather}
\frac{d \log \eta(\tau^{\RN}(t))}{d t}= \cN^{\RN} \frac{\eta_{\cN^{\RN}}^1 (t_{\ast}-t)}{2 \pi r}.\label{eqn:Dedekind_eta}
\end{gather}

\subsection{Transition probability densities of Brownian motions in an interval} \label{sec:appendixA_4}

Using the def\/initions of the Jacobi theta functions (\ref{eqn:theta}) and Jacobi's imaginary transforma\-tions~(\ref{eqn:Jacobi_imaginary}), we can obtain the following expressions for the transition probability densities $p^{[0, \pi r]}_{\rm aa}(t, y|x)$ and $p^{[0, \pi r]}_{\rm rr}(t, y|x)$ def\/ined by (\ref{eqn:p_abs}) and (\ref{eqn:p_ref})
\begin{align}
p^{[0, \pi r]}_{\rm aa}(t, y|x)&= p_{\rm BM}(t, y|x) \vartheta_3 \left( \frac{i (y-x) r}{t}; \frac{2 \pi i r^2}{t} \right)
- p_{\rm BM}(t, y|-x) \vartheta_3 \left( \frac{i (y+x) r}{t}; \frac{2 \pi i r^2}{t} \right)
\nonumber\\
&= \frac{1}{2 \pi r}
\left\{ \vartheta_3 \left( \frac{y-x}{2 \pi r}; \frac{it}{2 \pi r^2} \right)
- \vartheta_3 \left( \frac{y+x}{2 \pi r}; \frac{it}{2 \pi r^2} \right) \right\}
\nonumber\\
&= \frac{1}{\pi r} \sum_{n \in \Z} e^{-n^2 t/2r^2}
\sin \left(\frac{ny}{r} \right) \sin \left( \frac{nx}{r} \right),
\label{eqn:p_abs_A}
\end{align}
and
\begin{align}
p^{[0, \pi r]}_{\rm rr}(t, y|x)
&= p_{\rm BM}(t, y|x)
\vartheta_3 \left( \frac{i (y-x) r}{t}; \frac{2 \pi i r^2}{t} \right)
+ p_{\rm BM}(t, y|-x) \vartheta_3 \left( \frac{i (y+x) r}{t}; \frac{2 \pi i r^2}{t} \right)
\nonumber\\
&= \frac{1}{2 \pi r}
\left\{ \vartheta_3 \left( \frac{y-x}{2 \pi r}; \frac{it}{2 \pi r^2} \right)
+ \vartheta_3 \left( \frac{y+x}{2 \pi r}; \frac{it}{2 \pi r^2} \right) \right\}
\nonumber\\
&= \frac{1}{\pi r} \sum_{n \in \Z} e^{-n^2 t/2r^2}
\cos \left(\frac{ny}{r} \right) \cos \left( \frac{nx}{r} \right),
\label{eqn:p_ref_A}
\end{align}
for $x, y \in [0, \pi r]$, $t \geq 0$. It should be noted that $\int_0^{\pi r} p^{[0, \pi r]}_{\rm rr}(t, y|u) dy \equiv 1$, $\forall\, t \geq 0$, $u \in [0, \pi r]$, while
\begin{gather*}
p^{\rm survival}_u(t) \equiv \int_0^{\pi r} p^{[0, \pi r]}_{\rm aa}(t, y|u) dy
\simeq \frac{4}{\pi} e^{-t/2r^2} \sin \left( \frac{u}{r} \right) \to 0 \qquad \text{as $t \to \infty$}, \qquad u \in (0, \pi r).
\end{gather*}
When both boundaries at 0 and $\pi r$ are absorbing, $p^{\rm survival}_u(t)$ gives the probability for a Brownian motion to survive up to time $t$ in the interval $(0, \pi r)$, when it starts from $u \in (0, \pi r)$. In other words, with probability $1-p^{\rm survival}_u(t)$, the Brownian motion was absorbed before time $t$ and has been f\/ixed at one of the boundaries.

\subsection[Weierstrass $\wp$ function and zeta function $\zeta$]{Weierstrass $\boldsymbol{\wp}$ function and zeta function $\boldsymbol{\zeta}$}\label{sec:appendixA_5}

The Weierstrass $\wp$ function and zeta function $\zeta$ are def\/ined by
\begin{gather}
\wp(z| 2 \omega_1, 2 \omega_3)= \frac{1}{z^2}+\sum_{(m,n) \in \Z^2 \setminus \{(0,0)\}}\left[ \frac{1}{(z-\Omega_{m,n})^2}-\frac{1}{{\Omega_{m,n}}^2} \right],
\nonumber\\
\zeta(z| 2 \omega_1, 2 \omega_3)= \frac{1}{z}+\sum_{(m,n) \in \Z^2 \setminus \{(0,0)\}}\left[ \frac{1}{z-\Omega_{m,n}}+\frac{1}{\Omega_{m,n}}
+\frac{z}{{\Omega_{m,n}}^2} \right],\label{eqn:wp_zeta}
\end{gather}
where $\omega_1$ and $\omega_3$ are fundamental periods with $\tau=\omega_3/\omega_1$ and $\Omega_{m,n}=2 m \omega_1 + 2 n \omega_3$. Put~(\ref{eqn:Weierstrass1}), then the following functional equation holds (see \cite[Section~20.41]{WW27} and \cite[Lemma~2.1]{Kat15})
\begin{gather}
( \zeta_{\cN^{\RN}}(t_{\ast}-s, z+u) - \zeta_{\cN^{\RN}}(t_{\ast}-s, z) -\zeta_{\cN^{\RN}}(t_{\ast}-s, u) )^2\nonumber\\
 \qquad{} =\wp_{\cN^{\RN}}(t_{\ast}-s, z+u)+\wp_{\cN^{\RN}}(t_{\ast}-s, z)+\wp_{\cN^{\RN}}(t_{\ast}-s, u).\label{eqn:zeta_wp}
\end{gather}

\subsection*{Acknowledgements}

The author would like to thank the anonymous referees whose comments considerably improved the presentation of the paper. A part of the present work was done during the participation of the author in the ESI workshop on ``Elliptic Hypergeometric Functions in Combinatorics, Integrable Systems and Physics'' (March 20--24, 2017). The present author expresses his gratitude for the hospitality of Erwin Schr\"odinger International Institute for Mathematics and Physics (ESI) of the University of Vienna and for well-organization of the workshop by Christian Krattenthaler, Masatoshi Noumi, Simon Ruijsenaars, Michael~J.~Schlosser, Vyacheslav~P.~Spiridonov, and S.~Ole Warnaar. He also thanks Soichi Okada, Masatoshi Noumi, Simon Ruijsenaars, and Michael~J.~Schlosser for useful discussion. This work was supported in part by the Grant-in-Aid for Scientif\/ic Research~(C) (No.~26400405), (B) (No.~26287019), and (S) (No.~16H06338) of Japan Society for the Promotion of Science.

\pdfbookmark[1]{References}{ref}
\LastPageEnding

\end{document}